\documentclass[10pt]{amsart}
\usepackage{geometry}
\usepackage{graphicx}
\usepackage{amssymb}
\usepackage{epstopdf}
\usepackage[colorlinks]{hyperref}
\usepackage[lite]{amsrefs}
\usepackage{mathpazo}
\usepackage{amsmath}
\usepackage{mathrsfs}
\usepackage{extarrows}
\usepackage{cleveref}
\newtheorem{theorem}{Theorem}[section]
\newtheorem{corollary}[theorem]{Corollary}
\newtheorem{lemma}[theorem]{Lemma}
\newtheorem{proposition}[theorem]{Proposition}
\theoremstyle{definition}
\newtheorem{definition}[theorem]{Definition}
\newtheorem{remark}[theorem]{Remark}

\numberwithin{equation}{section}

\usepackage{tikz-cd}
\usepackage{tikz}
\usetikzlibrary{knots}
\usetikzlibrary{calc}
\usepgfmodule{decorations}
\usetikzlibrary{decorations.markings}
\usetikzlibrary{hobby}

\title{The necessity of (co)unit in nearly Frobenius algebra}
\author{Zhiyun Cheng}
\address{School of Mathematical Sciences, Beijing Normal University, Beijing 100875, China}
\email{czy@bnu.edu.cn}
\author{Ziyi Lei}
\address{School of Mathematical Sciences, Beijing Normal University, Beijing 100875, China}
\email{202121130067@mail.bnu.edu.cn}
\begin{document}
\begin{abstract}
In this article, we concern the concept of nearly Frobenius algebra, which corresponds to most 2D-TQFT of which each cobordism admits no critical points of index 0 or 2. We prove that any nearly Frobenius algebra over a principal ideal domain with surjective multiplication and injective comultiplication is indeed a Frobenius algebra. The motivation of this study mainly emanates from the investigation of potential constructions of link homology.
\end{abstract}
\maketitle

\section{Introduction}\label{section1}
A topological quantum field theory (TQFT) is a quantum field theory which appears in the description of physical systems such as quantum gravity and quantum Hall effect. It was first formulated by Witten in \cite{Wit1988} and soon later an axiomatic approach to it was proposed by Atiyah in \cite{Ati1989}. Roughly speaking, an $n$-dimensional TQFT ($n$D-TQFT) is a symmetric monoidal functor from the cobordism category $\textbf{nCob}$ to the category $\textbf{Veck}_{\mathbb{K}}$, where $\textbf{Veck}_{\mathbb{K}}$ denotes the symmetric monoidal category of $\mathbb{K}$-vector spaces equipped with tensor product. When $n=2$, 2D-TQFTs have a comparatively simple classification: there exists a canonical equivalence between the category of 2-dimensional TQFTs and the category of commutative Frobenius algebras \cite{Abr1996, Kock2004}.

In this article, we concern a specific class of 2D-TQFTs, where each cobordism admits no \emph{birth cobordism} or \emph{death cobordism}. Equivalently speaking, there is no critical points of index 0 or 2 in each cobordism. This kind of 2D-TQFT is called an \emph{almost 2D-TQFT} and the corresponding Frobenius algebra is named as \emph{nearly Frobenius algebra} \cite{GLSU2019}. In general, it is not necessary that a nearly Frobenius algebra must have a unit or counit. However, as the main result of this article, we prove the following result.

\begin{theorem}\label{theorem1.1}
Any 2-dimensional nearly Frobenius algebra over a principal ideal domain with surjective multiplication and injective comultiplication is a Frobenius algebra.
\end{theorem} 

The key ingredient of the proof of Theorem \ref{theorem1.1} is a detailed analysis of 2-dimensional commutative associative algebras. There are several papers which concerns the question of classification of 2-dimensional algebras, see \cite{EM2017, HUI2020, JJJ2007}. Based on these results, we discuss when a 2-dimensional commutative algebra over a field is associative and unital and then extend the result to any principal ideal domain (PID) coefficients. The key result can be summarised as follows.

\begin{theorem}\label{theorem1.2}
Any 2-dimensional commutative associative algebra over a PID with surjective multiplication is unital.
\end{theorem}

The reader may feel curious about the restrictive conditions such as 2-dimensional algebra and surjective multiplication mentioned above. The motivation comes from the construction of link homology. In \cite{Kho2000}, Khovanov introduced a bigraded homology group to oriented links in $\mathbb{R}^3$ such that the Euler characteristic is equal to the Jones polynomial. To define the Khovanov homology, one needs to consider all the resolutions of a link diagram and connect adjacent resolutions by cobordisms. Restrict our eyesight to the case that uses 0, 1-resolutions and 2D-TQFTs, such a theory is decided by the choice of the corresponding Frobenius algebra. Khovanov's original homology is based on the Frobenius algebra $\mathbb{Z}[1,x]/(x^2)$ with trace $\epsilon: 1\mapsto 0,x\mapsto 1$.  Soon after Khovanov's seminal work, Lee constructed the first variant of Khovanov homology by using $\mathbb{Z}[1,x]/(x^2-1)$ with the same trace \cite{Lee2005}. Although Lee homology contains only the information of number of components of the link, it possesses a filtration which was later used by Rasmussen to give a lower bound for slice genus and an alternative proof of Milnor conjecture \cite{Ra2010}. Comparing with the first proof given by Kronheimer and Mrowka by using gauge theory \cite{KM1993}, Rasmussen's proof is combinatorial. In \cite{Kho2006}, Khovanov showed that any such link homology can be realized by the Frobenius algebra $A_5=\mathbb{Z}[h,t]/(x^2-hx-t)$ with the same trace.

However, in the construction of Khovanov-type homology, only the multiplication and comultiplication get used. Hence a natural question arises: if a Frobenius algebra contains no unit or counit, whether it can also be used to define a link homology? On the aspect of 2D-TQFT, it means to erase the birth and death cobordisms from the generators. As a corollary of Theorem \ref{theorem1.1}, we unfortunately find that if such a nearly Frobenius algebra can be used to define a link homology, the unit and counit must exist. This statement still holds even if we extend the coefficients from fields to principal ideal domains.

This article is arranged as follows. In Section \ref{almost 2D-TQFT}, we introduce the notion of almost 2D-TQFT and explore the corresponding algebraic structure, the nearly Frobenius algebra. In Section \ref{Rank-$2$ algebras}, we provide a detailed analysis of 2-dimensional commutative algebras and prove the existence of the unit provided that the multiplication is surjective. Section \ref{section4} is devoted to discuss the case of noncommutative algebras. In Section \ref{Homology induced by almost 2D-TQFT}, we explain how an almost 2D-TQFT induce a diagram homology, review Khovanov's universal theory briefly, and finally deduce that any almost 2D-TQFT over PID that defines a link homology can be covered by $A_5$.

\section{almost 2D-TQFT and nearly Frobenius algebra}\label{almost 2D-TQFT}
Recall that a $2$D-TQFT is a functor from the category of $2$-dimensional cobordisms to the tensor (monoidal) category whose objects are the tensor products of finitely many fixed modules \cite{Kock2004}, and these theories are one-to-one corresponding to commutative Frobenius algebras. In this section, we first review the concept of almost 2D-TQFT, which, similar to a 2D-TQFT, can be regarded as a functor from a subcategory of cobordisms to the tensor category of fixed modules.
 
\subsection{The definition}
Before we put forward the definition of almost 2D-TQFT, we need to clarify which category it acts on first.

\begin{definition}
The \emph{nearly $2$-dimensional cobordism category} is a category whose objects are non-empty 1-dimensional oriented manifolds and morphisms are oriented cobordisms with each component has no critical points of index 0 or 2, quotient by diffeomorphisms relative to boundaries.
\end{definition}

As a subcategory of the $2$-dimensional cobordism category $\textbf{2Cob}$, let us use $\textbf{N2Cob}$ to denote the nearly 2-dimensional cobordism category. Then the category $\textbf{N2Cob}$ can be regarded as the category obtained from $\textbf{2Cob}$ by forsaking the empty objects and the morphisms containing components with locally maximal point or minimal point. For example, in Figure \ref{fig:Examples of morphisms}, the first cobordism is a morphism of $\textbf{N2Cob}$ but the second and the third one are not. In fact, the morphisms in $\mathrm{Mor}(\textbf{N2Cob})$ are the disjoint union of some cobordisms as the first one, which consists of pants and cylinders with non-empty boundaries on each side.

\begin{figure}[htbp]
    \centering
    \begin{tikzpicture}[scale=0.6, use Hobby shortcut]
      \draw (0,1) ellipse (0.2 and 0.5);
      \draw (0,-1) ellipse (0.2 and 0.5);
      \node at (0,-1.85) {$\vdots$};
      \draw (0,-3) ellipse (0.2 and 0.5);
      \draw (0,1.5)..++(0.01,0)..(2,0.5)..++(0.01,0);
      \draw (0,-1.5)..++(0.01,0)..(2,-0.5)..++(0.01,0);
      \draw (0,-2.5)..++(0.01,0)..(2,-1.5)..++(0.01,0);
      \draw (0,-3.5)..++(0.01,0)..(2,-2.5)..++(0.01,0);
      \draw (0,-0.5) arc (-90:90:0.7 and 0.5);
      \node at (2,-0.85) {$\vdots$};
      \draw (3-0.1,-1+0.1)..(3,-1)..++(0.01,-0.01)..(3.5,-1.2)..(4,-1)..++(0.01,0.01)..++(0.09,0.09);
      \draw (3,-1)..++(0.01,0.01)..(3.5,-0.8)..(4,-1)..++(0.01,-0.01);
      \node at (5,-1) {$\cdots$};
      \draw (6-0.1,-1+0.1)..(6,-1)..++(0.01,-0.01)..(6.5,-1.2)..(7,-1)..++(0.01,0.01)..++(0.09,0.09);
      \draw (6,-1)..++(0.01,0.01)..(6.5,-0.8)..(7,-1)..++(0.01,-0.01);
      \draw (10,1) ellipse (0.2 and 0.5);
      \draw (10,-1) ellipse (0.2 and 0.5);
      \node at (10,-1.85) {$\vdots$};
      \draw (10,-3) ellipse (0.2 and 0.5);
      \draw (2,0.5)--(8,0.5) (2,-2.5)--(8,-2.5);
      \draw (8,0.5)..++(0.01,0)..(10,1.5)..++(0.01,0);
      \draw (8,-0.5)..++(0.01,0)..(10,-1.5)..++(0.01,0);
      \draw (8,-1.5)..++(0.01,0)..(10,-2.5)..++(0.01,0);
      \draw (8,-2.5)..++(0.01,0)..(10,-3.5)..++(0.01,0);
      \node at (8,-0.85) {$\vdots$};
      \draw (10,0.5) arc (90:270:0.7 and 0.5);
    \end{tikzpicture}
    \quad
    \begin{tikzpicture}[scale=0.9, use Hobby shortcut]
      \draw (15,1.5)..++(-0.01,0)..(12,-0.5)..++(-0.01,0);
      \draw (15,-1.5)--(12,-1.5);
      \draw (15,0.5) arc (90:270:0.7 and 0.5);
      \draw (12,0.5) arc (-90:90:0.7 and 0.5);
      \draw (15,1) ellipse (0.2 and 0.5);
      \draw (15,-1) ellipse (0.2 and 0.5);
      \draw (12,-1) ellipse (0.2 and 0.5);
      \draw (12,1) ellipse (0.2 and 0.5);
    \end{tikzpicture}
    \quad \quad 
    \begin{tikzpicture}[scale=0.9, use Hobby shortcut]
        \draw (15,1.5) arc (90:270:2.1 and 1.5);
        \draw (15,0.5) arc (90:270:0.7 and 0.5);
        \draw (15,1) ellipse (0.2 and 0.5);
        \draw (15,-1) ellipse (0.2 and 0.5);
      \end{tikzpicture}
    \caption{Examples of morphisms in $\textbf{2Cob}$}
    \label{fig:Examples of morphisms}
\end{figure}
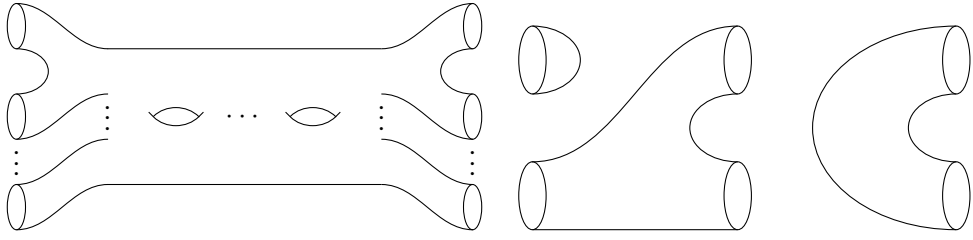
\par

Now we can define what an almost 2D-TQFT is.
\begin{definition}
Let $R$ be a commutative ring with identity element. An \emph{almost 2D-TQFT} is a symmetric monoidal functor from the category $\textbf{N2Cob}$ to the tensor category of $R$-module $A$. It maps each object of $\textbf{N2Cob}$, say $\sqcup_mS^1$, to $\otimes_mA$ and maps each cobordism to a linear $R$-homomorphism, where disjoint unions of cobordisms are mapped to the tensor products of the image of each connected component.
\end{definition}

\subsection{Relative algebraic structure}
Recall that the morphisms of a \textbf{2Cob} can be generated by six basic cobordisms, the birth and death of one circle, the pant that sends one circle to twos, the pant that merges two circles into one, the cylinder of one circle, and the permutation of two circles, see Figure \ref{fig: Basic cobordisms and corresponding operators}. In the case of \textbf{N2Cob}, the necessary basic cobordisms are pairs of pants, cylinders and permutations. This is because each connected component of the morphisms in $\mathrm{Mor}(\textbf{N2Cob})$, up to permutations, is as shown by the first picture of Figure \ref{fig:Examples of morphisms}, which can be decomposed into pairs of pants and cylinders. Note that a permutation is just a disjoint union of two cylinders.

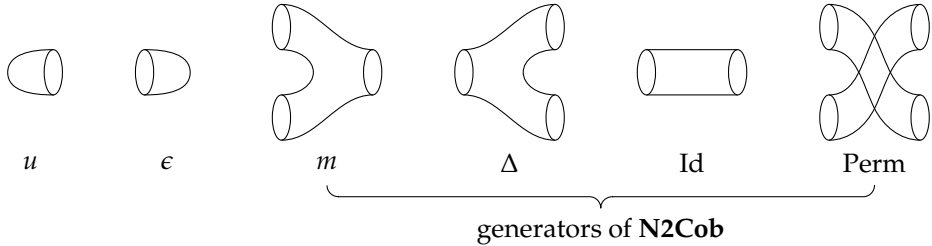
\begin{figure}[htbp]
  \center 
  \begin{tikzpicture}[scale=0.6, use Hobby shortcut]
    \draw (1,0.5)..++(-0.01,0)..(0,0)..++(0,-0.01)..(1,-0.5)..++(0.01,0);
    \draw (1,0) ellipse (0.2 and 0.5);
    \node at (0.5,-2) {$u$};
    \draw (3,0.5)..++(0.01,0)..(4,0)..++(0,-0.01)..(3,-0.5)..++(-0.01,0);
    \draw (3,0) ellipse (0.2 and 0.5);
    \node at (3.5,-2) {$\epsilon$};
    \draw (6+2+4,1.5)..++(-0.01,0)..(4+2+4,0.5)..++(-0.01,0);
    \draw (6+2+4,-1.5)..++(-0.01,0)..(4+2+4,-0.5)..++(-0.01,0);
    \draw (6+2+4,0.5) arc (90:270:0.7 and 0.5);
    \draw (6+2+4,1) ellipse (0.2 and 0.5);
    \draw (6+2+4,-1) ellipse (0.2 and 0.5);
    \draw (4+2+4,0) ellipse (0.2 and 0.5);
    \node at (5+2+4,-2) {$\Delta$};
    \draw (10-4,1.5)..++(0.01,0)..(12-4,0.5)..++(0.01,0);
    \draw (10-4,-1.5)..++(0.01,0)..(12-4,-0.5)..++(0.01,0);
    \draw (10-4,-0.5) arc (-90:90:0.7 and 0.5);
    \draw (10-4,1) ellipse (0.2 and 0.5);
    \draw (10-4,-1) ellipse (0.2 and 0.5);
    \draw (12-4,0) ellipse (0.2 and 0.5);
    \node at (11-4,-2) {$m$};
    \draw (16-2,0.5)..(18-2,0.5) (16-2,-0.5)..(18-2,-0.5);
    \draw (18-2,0) ellipse (0.2 and 0.5);
    \draw (16-2,0) ellipse (0.2 and 0.5);
    \node at (17-2,-2) {$\mathrm{Id}$};
    \draw (18,1) ellipse (0.2 and 0.5);
    \draw (18,-1) ellipse (0.2 and 0.5);
    \draw (20,1) ellipse (0.2 and 0.5);
    \draw (20,-1) ellipse (0.2 and 0.5);
    \draw (18,1.5)..++(0.01,0)..(19,0.8)..(20,-0.5)..++(0.01,0);
    \draw (18,0.5)..++(0.01,0)..(19,-0.8)..(20,-1.5)..++(0.01,0);
    \draw (18,-0.5)..++(0.01,0)..(19,0.8)..(20,1.5)..++(0.01,0);
    \draw (18,-1.5)..++(0.01,0)..(19,-0.8)..(20,0.5)..++(0.01,0);
    \node at (19,-2) {$\mathrm{Perm}$};
    \draw (7,-2.5)..++(0,-0.01)..(7.25,-2.75)..++(0.01,0) (7.25,-2.75)..(12.75,-2.75) (12.75,-2.75)..++(0.01,0)..(13,-3)..++(0,-0.01) (13,-3)..++(0,0.01)..(13.25,-2.75)..++(0.01,0) (13.25,-2.75)..(18.75,-2.75) (18.75,-2.75)..++(0.01,0)..(19,-2.5)..++(0,0.01);
    \node at (13,-3.5) {generators of $\textbf{N2Cob}$};
  \end{tikzpicture}
  \caption{Basic cobordisms and corresponding operators}
  \label{fig: Basic cobordisms and corresponding operators}
\end{figure}

Under the action of almost 2D-TQFT, these four basic cobordisms $m, \Delta, \mathrm{Id}, \mathrm{Perm}$ become four linear $R$-homomorphisms between tensor products of the $R$-module $A$. The pant merging two circles into one becomes the multiplication $m: A\otimes A\to A$, the pant sending one circle to twos becomes the comultiplication $\Delta: A\to A\otimes A$, the annulus becomes the identity map $\mathrm{Id}: A\to A$ and the permutation becomes the map $\mathrm{Perm}: A\otimes A\to A\otimes A, \sum x_i\otimes y_i\mapsto\sum y_i\otimes x_i$. As a subcategory of $\textbf{2Cob}$, up to permutations, the category $\textbf{N2Cob}$ inherits most important properties of $\textbf{2Cob}$, which are illustrated in Figure \ref{fig:Diffeomorphism cobordisms}.

\begin{figure}
  \begin{tikzpicture}[scale=0.6, use Hobby shortcut, baseline=-20pt]
    \draw (0,1) ellipse (0.2 and 0.5);
    \draw (0,-1) ellipse (0.2 and 0.5);
    \draw (0,-3) ellipse (0.2 and 0.5);
    \draw (0,1.5)..++(0.01,0)..(2,0.5)..++(0.01,0);
    \draw (0,-1.5)..++(0.01,0)..(2,-0.5)..++(0.01,0);
    \draw (0,-2.5)..++(0.01,0)..(2,-1.5)..++(0.01,0);
    \draw (0,-3.5)..++(0.01,0)..(2,-2.5)..++(0.01,0);
    \draw (0,-0.5) arc (-90:90:0.7 and 0.5);
    \draw (2,0) ellipse (0.2 and 0.5);
    \draw (2,-2) ellipse (0.2 and 0.5);
    \draw (2,-1.5) arc (-90:90:0.7 and 0.5);
    \draw (2,0.5)..++(0.01,0)..(4,-0.5)..++(0.01,0);
    \draw (2,-2.5)..++(0.01,0)..(4,-1.5)..++(0.01,0);
    \draw (4,-1) ellipse (0.2 and 0.5);
  \end{tikzpicture}
  \quad $\xlongequal{\textbf{diff.}}$ \quad
  \begin{tikzpicture}[scale=0.6, use Hobby shortcut, baseline=-20pt]
    \draw (0,1) ellipse (0.2 and 0.5);
    \draw (0,-1) ellipse (0.2 and 0.5);
    \draw (0,-3) ellipse (0.2 and 0.5);
    \draw (0,1.5)..++(0.01,0)..(2,0.5)..++(0.01,0);
    \draw (0,0.5)..++(0.01,0)..(2,-0.5)..++(0.01,0);
    \draw (0,-0.5)..++(0.01,0)..(2,-1.5)..++(0.01,0);
    \draw (0,-3.5)..++(0.01,0)..(2,-2.5)..++(0.01,0);
    \draw (0,-2.5) arc (-90:90:0.7 and 0.5);
    \draw (2,0) ellipse (0.2 and 0.5);
    \draw (2,-2) ellipse (0.2 and 0.5);
    \draw (2,-1.5) arc (-90:90:0.7 and 0.5);
    \draw (2,0.5)..++(0.01,0)..(4,-0.5)..++(0.01,0);
    \draw (2,-2.5)..++(0.01,0)..(4,-1.5)..++(0.01,0);
    \draw (4,-1) ellipse (0.2 and 0.5);
  \end{tikzpicture}
  \quad\quad 
  \begin{tikzpicture}[scale=0.6, use Hobby shortcut, baseline=-5pt]
    \draw (0,1) ellipse (0.2 and 0.5);
    \draw (0,-1) ellipse (0.2 and 0.5);
    \draw (0,1.5)..++(0.01,0)..(1,0.8)..(2,-0.5)..++(0.01,0);
    \draw (0,0.5)..++(0.01,0)..(1,-0.8)..(2,-1.5)..++(0.01,0);
    \draw (0,-1.5)..++(0.01,0)..(1,-0.8)..(2,0.5)..++(0.01,0);
    \draw (0,-0.5)..++(0.01,0)..(1,0.8)..(2,1.5)..++(0.01,0);
    \draw (2,1) ellipse (0.2 and 0.5);
    \draw (2,-0.5) arc (-90:90:0.7 and 0.5);
    \draw (2,-1) ellipse (0.2 and 0.5);
    \draw (2,1.5)..++(0.01,0)..(4,0.5)..++(0.01,0);
    \draw (2,-1.5)..++(0.01,0)..(4,-0.5)..++(0.01,0);
    \draw (4,0) ellipse (0.2 and 0.5);
  \end{tikzpicture}
  \quad $\xlongequal{\textbf{diff.}}$ \quad
  \begin{tikzpicture}[scale=0.6, use Hobby shortcut, baseline=-5pt]
    \draw (2,1) ellipse (0.2 and 0.5);
    \draw (2,-0.5) arc (-90:90:0.7 and 0.5);
    \draw (2,-1) ellipse (0.2 and 0.5);
    \draw (2,1.5)..++(0.01,0)..(4,0.5)..++(0.01,0);
    \draw (2,-1.5)..++(0.01,0)..(4,-0.5)..++(0.01,0);
    \draw (4,0) ellipse (0.2 and 0.5);
  \end{tikzpicture}
  \\
  \begin{tikzpicture}[scale=0.6, use Hobby shortcut, baseline=-20pt]
    \draw (0,1) ellipse (0.2 and 0.5);
    \draw (0,-1) ellipse (0.2 and 0.5);
    \draw (0,-3) ellipse (0.2 and 0.5);
    \draw (0,1.5)..++(-0.01,0)..(-2,0.5)..++(-0.01,0);
    \draw (0,-1.5)..++(-0.01,0)..(-2,-0.5)..++(-0.01,0);
    \draw (0,-2.5)..++(-0.01,0)..(-2,-1.5)..++(-0.01,0);
    \draw (0,-3.5)..++(-0.01,0)..(-2,-2.5)..++(-0.01,0);
    \draw (0,0.5) arc (90:270:0.7 and 0.5);
    \draw (-2,0) ellipse (0.2 and 0.5);
    \draw (-2,-2) ellipse (0.2 and 0.5);
    \draw (-2,-0.5) arc (90:270:0.7 and 0.5);
    \draw (-2,0.5)..++(-0.01,0)..(-4,-0.5)..++(-0.01,0);
    \draw (-2,-2.5)..++(-0.01,0)..(-4,-1.5)..++(-0.01,0);
    \draw (-4,-1) ellipse (0.2 and 0.5);
  \end{tikzpicture}
  \quad $\xlongequal{\textbf{diff.}}$ \quad
  \begin{tikzpicture}[scale=0.6, use Hobby shortcut, baseline=-20pt]
    \draw (0,1) ellipse (0.2 and 0.5);
    \draw (0,-1) ellipse (0.2 and 0.5);
    \draw (0,-3) ellipse (0.2 and 0.5);
    \draw (0,1.5)..++(-0.01,0)..(-2,0.5)..++(-0.01,0);
    \draw (0,0.5)..++(-0.01,0)..(-2,-0.5)..++(-0.01,0);
    \draw (0,-0.5)..++(-0.01,0)..(-2,-1.5)..++(-0.01,0);
    \draw (0,-3.5)..++(-0.01,0)..(-2,-2.5)..++(-0.01,0);
    \draw (0,-1.5) arc (90:270:0.7 and 0.5);
    \draw (-2,0) ellipse (0.2 and 0.5);
    \draw (-2,-2) ellipse (0.2 and 0.5);
    \draw (-2,-0.5) arc (90:270:0.7 and 0.5);
    \draw (-2,0.5)..++(-0.01,0)..(-4,-0.5)..++(-0.01,0);
    \draw (-2,-2.5)..++(-0.01,0)..(-4,-1.5)..++(-0.01,0);
    \draw (-4,-1) ellipse (0.2 and 0.5);
  \end{tikzpicture}
  \quad\quad 
  \begin{tikzpicture}[scale=0.6, use Hobby shortcut, baseline=-5pt]
    \draw (0,1) ellipse (0.2 and 0.5);
    \draw (0,-1) ellipse (0.2 and 0.5);
    \draw (0,1.5)..++(-0.01,0)..(-1,0.8)..(-2,-0.5)..++(-0.01,0);
    \draw (0,0.5)..++(-0.01,0)..(-1,-0.8)..(-2,-1.5)..++(-0.01,0);
    \draw (0,-1.5)..++(-0.01,0)..(-1,-0.8)..(-2,0.5)..++(-0.01,0);
    \draw (0,-0.5)..++(-0.01,0)..(-1,0.8)..(-2,1.5)..++(-0.01,0);
    \draw (-2,1) ellipse (0.2 and 0.5);
    \draw (-2,0.5) arc (90:270:0.7 and 0.5);
    \draw (-2,-1) ellipse (0.2 and 0.5);
    \draw (-2,1.5)..++(-0.01,0)..(-4,0.5)..++(-0.01,0);
    \draw (-2,-1.5)..++(-0.01,0)..(-4,-0.5)..++(-0.01,0);
    \draw (-4,0) ellipse (0.2 and 0.5);
  \end{tikzpicture}
  \quad $\xlongequal{\textbf{diff.}}$ \quad
  \begin{tikzpicture}[scale=0.6, use Hobby shortcut, baseline=-5pt]
    \draw (-2,1) ellipse (0.2 and 0.5);
    \draw (-2,0.5) arc (90:270:0.7 and 0.5);
    \draw (-2,-1) ellipse (0.2 and 0.5);
    \draw (-2,1.5)..++(-0.01,0)..(-4,0.5)..++(-0.01,0);
    \draw (-2,-1.5)..++(-0.01,0)..(-4,-0.5)..++(-0.01,0);
    \draw (-4,0) ellipse (0.2 and 0.5);
  \end{tikzpicture}
  \\
  \begin{tikzpicture}[scale=0.6, use Hobby shortcut, baseline=-20pt]
    \draw (0,1) ellipse (0.2 and 0.5);
    \draw (0,-1) ellipse (0.2 and 0.5);
    \draw (0,-3) ellipse (0.2 and 0.5);
    \draw (0,1.5)..++(-0.01,0)..(-2,0.5)..++(-0.01,0);
    \draw (0,0.5)..++(-0.01,0)..(-2,-0.5)..++(-0.01,0);
    \draw (0,-0.5)..++(-0.01,0)..(-2,-1.5)..++(-0.01,0);
    \draw (0,-3.5)..++(-0.01,0)..(-2,-2.5)..++(-0.01,0);
    \draw (0,-1.5) arc (90:270:0.7 and 0.5);
    \draw (-2,0) ellipse (0.2 and 0.5);
    \draw (-2,-2) ellipse (0.2 and 0.5);
    \draw (0,1.5)..++(0.01,0)..(2,0.5)..++(0.01,0);
    \draw (0,-1.5)..++(0.01,0)..(2,-0.5)..++(0.01,0);
    \draw (0,-2.5)..++(0.01,0)..(2,-1.5)..++(0.01,0);
    \draw (0,-3.5)..++(0.01,0)..(2,-2.5)..++(0.01,0);
    \draw (0,-0.5) arc (-90:90:0.7 and 0.5);
    \draw (2,0) ellipse (0.2 and 0.5);
    \draw (2,-2) ellipse (0.2 and 0.5);
  \end{tikzpicture}
  \quad $\xlongequal{\textbf{diff.}}$ \quad
  \begin{tikzpicture}[scale=0.6, use Hobby shortcut, baseline=-20pt]
    \draw (0,0) ellipse (0.2 and 0.5);
    \draw (0,-2) ellipse (0.2 and 0.5);
    \draw (2,-1) ellipse (0.2 and 0.5);
    \draw (4,0) ellipse (0.2 and 0.5);
    \draw (4,-2) ellipse (0.2 and 0.5);
    \draw (0,0.5)..++(0.01,0)..(2,-0.5)..++(0.01,0);
    \draw (0,-2.5)..++(0.01,0)..(2,-1.5)..++(0.01,0);
    \draw (2,-0.5)..++(0.01,0)..(4,0.5)..++(0.01,0);
    \draw (2,-1.5)..++(0.01,0)..(4,-2.5)..++(0.01,0);
    \draw (4,-0.5) arc (90:270:0.7 and 0.5);
    \draw (0,-1.5) arc (-90:90:0.7 and 0.5);
  \end{tikzpicture}
  \quad $\xlongequal{\textbf{diff.}}$ \quad
  \begin{tikzpicture}[scale=0.6, use Hobby shortcut, baseline=-20pt]
    \draw (0,1) ellipse (0.2 and 0.5);
    \draw (0,-1) ellipse (0.2 and 0.5);
    \draw (0,-3) ellipse (0.2 and 0.5);
    \draw (0,1.5)..++(-0.01,0)..(-2,0.5)..++(-0.01,0);
    \draw (0,-1.5)..++(-0.01,0)..(-2,-0.5)..++(-0.01,0);
    \draw (0,-2.5)..++(-0.01,0)..(-2,-1.5)..++(-0.01,0);
    \draw (0,-3.5)..++(-0.01,0)..(-2,-2.5)..++(-0.01,0);
    \draw (0,0.5) arc (90:270:0.7 and 0.5);
    \draw (-2,0) ellipse (0.2 and 0.5);
    \draw (-2,-2) ellipse (0.2 and 0.5);
    \draw (0,1.5)..++(0.01,0)..(2,0.5)..++(0.01,0);
    \draw (0,0.5)..++(0.01,0)..(2,-0.5)..++(0.01,0);
    \draw (0,-0.5)..++(0.01,0)..(2,-1.5)..++(0.01,0);
    \draw (0,-3.5)..++(0.01,0)..(2,-2.5)..++(0.01,0);
    \draw (0,-2.5) arc (-90:90:0.7 and 0.5);
    \draw (2,0) ellipse (0.2 and 0.5);
    \draw (2,-2) ellipse (0.2 and 0.5);
  \end{tikzpicture}\quad\quad\quad\quad\quad\quad\quad 
  \caption{Diffeomorphic cobordisms}
  \label{fig:Diffeomorphism cobordisms}
\end{figure}
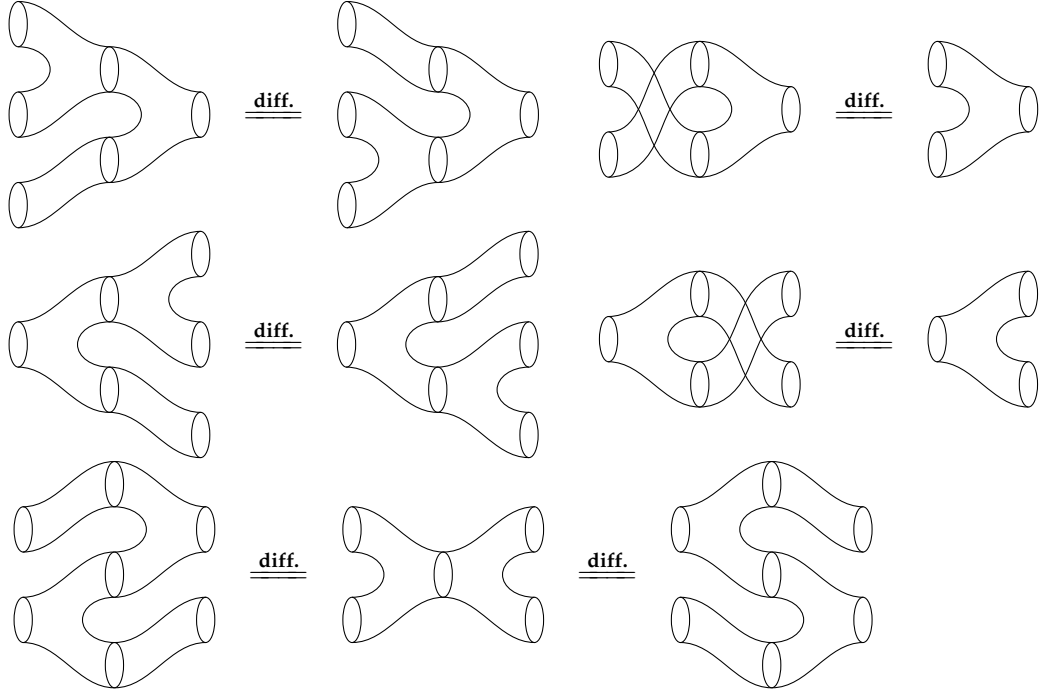

The first and second relations depicted in Figure~\ref{fig:Diffeomorphism cobordisms}, called \emph{associative relation} and \emph{commutative relation}, show that $(A, m)$ is a commutative and associative $R$-algebra\footnote{In this article, an algebra may not have an identity element. If it has an identity element, we say it is \emph{unital}.}. The third and fourth relations, named \emph{co-associative relation} and \emph{co-commutative relation}, illustrate that the comultiplication $\Delta$ is co-associative and co-communicative. And the last relation, usually called the \emph{Frobenius relation}, is equivalent to the statement that $\Delta$ is an algebraic homomorphism. As an algebraic structure induced by almost 2D-TQFT, it is natural to introduce the definition of nearly Frobenius algebra.

\begin{definition}[\cite{AGL2015, GLSU2019}]
A \emph{nearly Frobenius algebra} is a commutative, associative $R$-algebra $(A, m)$, which may not contain identity element, equipped with a co-associative and co-commutative homomorphism $\Delta\in\mathrm{Hom}(A, A\otimes A)$.
\end{definition}

\begin{remark}
As pointed out in \cite{AGL2015}, if $A$ is unital then the comultiplication is completely determined by its evaluation on $1\in A$, since $\Delta(x)=(x\otimes1)\Delta(1)$.
\end{remark}

Similar to the fact that there exists a one-to-one correspondence between $2$D-TQFTs and Frobenius algebras, we have the following result.
 
\begin{proposition}[\cite{GLSU2019}]
There exists a one-to-one correspondence between almost 2D-TQFTs and nearly Frobenius algebras.
\end{proposition}
\begin{proof}
The proof is quite similar to the proof of the equivalence of categories between 2D-TQFTs and Frobenius algebras, see also \cite{GLSU2019}. We sketch it here.  According to the analysis above, now each almost 2D-TQFT induces a nearly Frobenius algebra. Conversely, each nearly Frobenius algebra provides a value for the image of each generator of $\mathrm{Mor}(\textbf{N2Cob})$ under the almost 2D-TQFT, and all the relations induced by the diffeomorphic cobordisms in Figure \ref{fig:Diffeomorphism cobordisms} are satisfied. All these relations, up to permutations, are sufficient by the similar reason as \cite[Section 1.4]{Kock2004}, which implies that each nearly Frobenius algebra induces an almost 2D-TQFT as well. It is easy to observe that these two maps are inverse of each other.
\end{proof}

Now the study of almost 2D-TQFTs can be translated into the language of nearly Frobenius algebras. We will focus on rank-$2$ algebras with principal integral domain coefficients in the next section.

\section{Rank-$2$ algebras}\label{Rank-$2$ algebras}
In this section, we classify all the rank-2 commutative algebras with field coefficients. In the end, we will give a proof of Theorem \ref{theorem1.2}.

\subsection{Field-coefficient case}
In this subsection, we first take a quick review of the classification theorem for commutative rank-$2$ algebras over an algebraically closed field of characteristic different from 2 \cite{EM2017}, then we provide a similar classification theorem with characteristic $2$. By investigating each isomorphic class of these algebras, we determine which of them are associative or unital.

\subsubsection{Characteristic other than $2$ case}
There are many works concerning the question of classifying rank-$2$ algebras, such as \cite{EM2017, HUI2020, JJJ2007}. Different authors describe their results by different representative algebras in each isomorphic class. Here we recall the result in \cite{EM2017} where the characteristic of coefficient ring does not equal to 2. When the field has characteristic rather than $2$, the main idea is to split the multiplication map into its symmetric part and skew-symmetric part. The skew-symmetric part, as being rank-$2$, is a Lie algebra. So the key point to classify the symmetric part, which can be done by counting the number of idempotent elements. Throughout this section we only consider the symmetric case.

\begin{theorem}\cite{EM2017}\label{Theorem in EM2017}
Any commutative rank-$2$ algebra over an algebraically closed field of characteristic different from $2$ is isomorphic to one of the following, here $e_1$ and $e_2$ denote the generators of the algebra.
  \begin{itemize}
    \setlength{\itemsep}{5pt}
    \setlength{\parsep}{5pt}
    \setlength{\parskip}{5pt}
    \item \(\left\{\begin{aligned}
      m^6(e_1,e_1)&=e_1, \\
      m^6(e_1,e_2)&=\alpha_2e_1+\beta_2e_2, \\
      m^6(e_2,e_2)&=e_2.
    \end{aligned}\right .\)
    \quad  
    \(\left\{\begin{aligned}
      m^7(e_1,e_1)&=e_1, \\
      m^7(e_1,e_2)&=e_1+\frac{1}{2}e_2, \\
      m^7(e_2,e_2)&=0. 
    \end{aligned}\right .\)
    \quad 
    \(\left\{\begin{aligned}
      m^8(e_1,e_1)&=e_1, \\
      m^8(e_1,e_2)&=\frac{1}{2}e_2, \\
      m^8(e_2,e_2)&=e_1. 
    \end{aligned}\right .\)
    \item \(\left\{\begin{aligned}
      m^9(e_1,e_1)&=e_1, \\
      m^9(e_1,e_2)&=\beta_2e_2, \\
      m^9(e_2,e_2)&=0.
    \end{aligned}\right .\)
     $\beta_2\neq\frac{1}{2}$, 
    \quad  
    \(\left\{\begin{aligned}
      m^{10}(e_1,e_1)&=e_1, \\
      m^{10}(e_1,e_2)&=e_1, \\
      m^{10}(e_2,e_2)&=\alpha_4e_1. 
    \end{aligned}\right .\)
    \quad 
    \(\left\{\begin{aligned}
      m^{11}(e_1,e_1)&=e_1, \\
      m^{11}(e_1,e_2)&=0, \\
      m^{11}(e_2,e_2)&=e_1. 
    \end{aligned}\right .\)
    \item \(\left\{\begin{aligned}
      m^{12}(e_1,e_1)&=e_1, \\
      m^{12}(e_1,e_2)&=0, \\
      m^{12}(e_2,e_2)&=0.
    \end{aligned}\right .\)
    \quad  
    \(\left\{\begin{aligned}
      m^{13}(e_1,e_1)&=e_2, \\
      m^{13}(e_1,e_2)&=e_2, \\
      m^{13}(e_2,e_2)&=0. 
    \end{aligned}\right .\)
    \quad 
    \(\left\{\begin{aligned}
      m^{14}(e_1,e_1)&=e_2, \\
      m^{14}(e_1,e_2)&=0, \\
      m^{14}(e_2,e_2)&=0. 
    \end{aligned}\right .\)
    \item \(\left\{\begin{aligned}
      m^{15}(e_1,e_1)&=e_2, \\
      m^{15}(e_1,e_2)&=-2e_1+3e_2, \\
      m^{15}(e_2,e_2)&=-8e_1+8e_2.
    \end{aligned}\right .\)
    \quad  
    \(\left\{\begin{aligned}
      m^{16}(e_1,e_1)&=0, \\
      m^{16}(e_1,e_2)&=e_1, \\
      m^{16}(e_2,e_2)&=0. 
    \end{aligned}\right .\)
    \quad 
    \(\left\{\begin{aligned}
      m^{17}(e_1,e_1)&=0, \\
      m^{17}(e_1,e_2)&=0, \\
      m^{17}(e_2,e_2)&=0. 
    \end{aligned}\right .\)
  \end{itemize}
  If the coefficient field $\mathbb{K}$ is not algebraically closed, we have also the following algebras where $\lambda_2\in\mathbb{K}\setminus(\mathbb{K}^*)^2$ and polynomial $P_A(y)$ equals to 
  \[-1+y(4\alpha_2+\beta_4)+y^2(2\alpha_4\beta_2-4\alpha_2^2-4\alpha_2\beta_4)+y^3(\alpha_4^2-4\alpha_2\alpha_4\beta_2+4\alpha_2^2\beta_4). \]
  \begin{itemize}
    \setlength{\itemsep}{5pt}
    \setlength{\parsep}{5pt}
    \setlength{\parskip}{5pt}
    \item \(\left\{\begin{aligned}
      m^{8,1}_R(e_1,e_1)&=e_1, \\
      m^{8,1}_R(e_1,e_2)&=\frac{1}{2}e_2, \\
      m^{8,1}_R(e_2,e_2)&=\lambda_2e_1.
    \end{aligned}\right .\)
    \,   
    \(\left\{\begin{aligned}
      m^{8,2}_R(e_1,e_1)&=e_1, \\
      m^{8,2}_R(e_1,e_2)&=\beta_2e_2, \\
      m^{8,2}_R(e_2,e_2)&=\lambda_2e_1. 
    \end{aligned}\right .\)
     $1-2\beta_2\notin(\mathbb{K}^*)^2$, 
    \(\left\{\begin{aligned}
      m^{11}_R(e_1,e_1)&=e_1, \\
      m^{11}_R(e_1,e_2)&=0, \\
      m^{11}_R(e_2,e_2)&=\lambda_2e_1. 
    \end{aligned}\right .\)
    \item \(\left\{\begin{aligned}
      m^{14,1}_R(e_1,e_1)&=e_1, \\
      m^{14,1}_R(e_1,e_2)&=\alpha_2e_1+e_2, \\
      m^{14,1}_R(e_2,e_2)&=0.
    \end{aligned}\right .\)
     $2\alpha_2+1\notin\mathbb{K}^2$, 
    \(\left\{\begin{aligned}
      m^{14,2}_R(e_1,e_1)&=e_1, \\
      m^{14,2}_R(e_1,e_2)&=\alpha_2e_1, \\
      m^{14,2}_R(e_2,e_2)&=0.
    \end{aligned}\right .\)
     $2\alpha_2+1\notin\mathbb{K}^2$, 
    \item \(\left\{\begin{aligned}
      m^{15,1}_R(e_1,e_1)&=e_2, \\
      m^{15,1}_R(e_1,e_2)&=\alpha_2e_1+\beta_2e_2, \\
      m^{15,1}_R(e_2,e_2)&=\alpha_4e_1+\beta_4e_2.
    \end{aligned}\right .\)
     $P_A(y)$ without roots. 
  \end{itemize}
\end{theorem}

\begin{remark}
Here we just repeat the statement of \cite[Theorem 6]{EM2017}. Actually, the case $m^{8, 1}_R$ can be recovered from $m^{8,2}_R$ by setting $\beta_2=\frac{1}{2}$.
\end{remark}

\begin{remark}\label{Idempotent remark}
According to the process of classification in \cite{EM2017}, the algebra(s) 
  \begin{itemize}
    \item $m^6$ has at least two idempotent elements $e_1$ and $e_2$;
    \item $m^7,\,m^8,\,m^{8,1}_R,\,m^{8,2}_R,\,m^9,\,m^{10},\,m^{11},\,m^{11}_R,\,m^{12}$ have only one idempotent $e_1$;
    \item $m^{13},\,m^{14},\,m^{14,1}_R,\,m^{14,2}_R,\,m^{15},\,m^{15,1}_R,\,m^{16},\,m^{17}$ have no idempotent. 
  \end{itemize} 
\end{remark}

Then we can check whether these algebras are associative or unital one by one. The result is listed in Proposition \ref{The further proposition with characteristic other than 2} below.
\begin{proposition}\label{The further proposition with characteristic other than 2}
For the algebras mentioned in Theorem \ref{Theorem in EM2017}, we have
  \begin{enumerate}
    \item $m^6$ is associative if and only if $(\alpha_2,\beta_2)\in\{(0, 0), (0, 1), (1, 0)\}$. These three isomorphic algebras are all unital;
    \item $m^7,m^8,m^{11},m^{15},m^{16},m^{8,1}_R,m^{14,1}_R,m^{14,2}_R$ are not associative nor unital;
    \item $m^9$ is associative if and only if $\beta_2=0$ or 1 and unital if and only if $\beta_2=1$;
    \item $m^{10}$ is associative if and only if $\alpha_4=1$, but it is not unital;
    \item $m^{12}, m^{13}, m^{14}, m^{17}$ are associative but not unital;
    \item $m^{8,2}_R$ is associative if and only if $\beta_2\in\{1, -1\}\nsubseteq(\mathbb{K}^*)^2$, at this point it is also unital;
    \item $m^{11}_R$ is associative if and only if $\lambda_2=0$, however it is not unital;
    \item $m^{15,1}_R$ is associative if and only if $\alpha_4=\alpha_2\beta_2,\,\beta_4=\alpha_2+\beta_2^2$, but it is not unital. 
  \end{enumerate}
\end{proposition}

Before we start our proof of Proposition~\ref{The further proposition with characteristic other than 2}, we need the following lemme. The first statement of it can be used to detect whether the algebra is associative and the second one shows that identity elements only exist in associative cases.

\begin{lemma}\label{Associative and unital lemma}
Suppose $(A, m)$ is a 2-dimensional commutative $R$-algebra generated by $e_1$ and $e_2$, which probably contains no identity element. Then 
\begin{enumerate}
\item it is associative if and only if $m(m(e_1,e_1),e_2)=m(e_1,m(e_1,e_2))$ and $m(m(e_2,e_2),e_1)=m(e_2,m(e_2,e_1))$;
\item it is associative if it contains an identity element. 
\end{enumerate}
\end{lemma}

\begin{proof}
  As the multiplication is linear, this algebra is associative if and only if $m(m(e_i, e_j), e_k)=m(e_i, m(e_j, e_k))$ for any $i, j, k\in\{1, 2\}$. When $i=j=k$, this equation follows from the commutativity of the multiplication. It suffices to consider the case that at least two of $\{i, j, k\}$ are distinct. When $i=k\neq j$, it holds by the commutativity as well. Then there are only four cases need to check, say
  \[m(m(e_1, e_1), e_2)=m(e_1, m(e_1, e_2)),\,m(m(e_2, e_2), e_1)=m(e_2, m(e_2, e_1)),\]
  \[m(m(e_2, e_1), e_1)=m(e_2, m(e_1, e_1)),\,m(m(e_1, e_2), e_2)=m(e_1, m(e_2, e_2)).\]
  Still, because $m$ is commutative, we know that the third equality is equivalent to the first one, and the fourth equality is equivalent to the second one. Then the equalities in the first statement is sufficient to determine whether corresponding algebra is associative or not.

If this algebra contains an identity element which is denoted by $1$, then $A$ can be generated by $1$ and another element, say $e_1$. Replacing $e_2$ with $1$ in the equalities in the first statement, we find that both of them hold. The second statement follows.
\end{proof}

Now we turn to the proof of Proposition \ref{The further proposition with characteristic other than 2}.

\begin{proof}
Let us check each case one by one.
\begin{enumerate}
  \item If $A=(\langle e_1,e_2\rangle,m^6)$: It can be calculated that 
  \begin{align*}
    m^6(m^6(e_1,e_1),e_2)&=\alpha_2e_1+\beta_2e_2,\\
    m^6(e_1,m^6(e_1,e_2))&=(\alpha_2+\alpha_2\beta_2)e_1+\beta_2^2e_2,\\
    m^6(m^6(e_2,e_2),e_1)&=\alpha_2e_1+\beta_2e_2,\\
    m^6(e_2,m^6(e_2,e_1))&=\alpha_2^2e_1+(\beta_2+\alpha_2\beta_2)e_2. 
  \end{align*}
Lemma~\ref{Associative and unital lemma} tells us that $A$ is associative if and only if $\alpha_2^2=\alpha_2$, $\beta_2^2=\beta_2$, $\alpha_2\beta_2=0$, which is equivalent to say, $(\alpha_2,\beta_2)\in\{(0, 0), (0, 1), (1, 0)\}$. For these cases, the algebra $A$ has identity element $e_1+e_2, e_1$ and $e_2$ respectively. And each of these three algebras is generated by an identity element and a nondegenerate idempotent, hence they are isomorphic to each other.
  \item If $A=(\langle e_1,e_2\rangle,m^7)$: It can be calculated that 
  \begin{align*}
    m^7(m^7(e_1,e_1),e_2)&=e_1+\frac{1}{2}e_2,\\
    m^7(e_1,m^7(e_1,e_2))&=\frac{3}{2}e_1+\frac{1}{4}e_2. 
  \end{align*}
  As $m^7(m^7(e_1,e_1),e_2)\neq m^7(e_1,m^7(e_1,e_2))$, by Lemma~\ref{Associative and unital lemma}, this algebra is not associative nor unital. 
  \item If $A=(\langle e_1,e_2\rangle,m^8)$: By the same reason of $m^7$, it is not associative nor unital either. 
  \item If $A=(\langle e_1,e_2\rangle,m^9)$: We have 
  \begin{align*}
    m^9(m^9(e_1,e_1),e_2)&=\beta_2e_2,\\
    m^9(e_1,m^9(e_1,e_2))&=\beta_2^2e_2,\\
    m^9(m^9(e_2,e_2),e_1)&=0,\\
    m^9(e_2,m^9(e_2,e_1))&=0. 
  \end{align*}
  So that the algebra is associative if and only if $\beta_2=0$ or1. And it can be easily verified that it is unital only if $\beta_2=1$, in which the identity element is $e_1$. 
  \item If $A=(\langle e_1,e_2\rangle,m^{10})$: Note that 
  \begin{align*}
    m^{10}(m^{10}(e_1,e_1),e_2)&=e_1,\\
    m^{10}(e_1,m^{10}(e_1,e_2))&=e_1,\\
    m^{10}(m^{10}(e_2,e_2),e_1)&=\alpha_4e_1,\\
    m^{10}(e_2,m^{10}(e_2,e_1))&=e_1,  
  \end{align*}
  It can be deduced from Lemma~\ref{Associative and unital lemma} that this algebra is associative if and only if $\alpha_4=1$. On the other hand, Remark \ref{Idempotent remark} suggests that $e_1$, as the only idempotent, is the only candidate of identity element as well. However, it is not an identity element because $m^{10}(e_1,e_2)=e_1\neq e_2$, hence this algebra is not unital.
  \item If $A=(\langle e_1,e_2\rangle,m^{11})$: We can also calculate that 
  \begin{align*}
    m^{11}(m^{11}(e_2,e_2),e_1)&=e_1,\\
    m^{11}(e_2,m^{11}(e_2,e_1))&=0\neq e_1.   
  \end{align*}
  So it is not associative, neither unital. 
  \item If $A=(\langle e_1,e_2\rangle,m^{12})$: because 
  \begin{align*}
    m^{12}(m^{12}(e_1,e_1),e_2)&=0,\\
    m^{12}(e_1,m^{12}(e_1,e_2))&=0,\\
    m^{12}(m^{12}(e_2,e_2),e_1)&=0,\\
    m^{12}(e_2,m^{12}(e_2,e_1))&=0,  
  \end{align*}
  we obtain that this algebra is associative. As it contains only one idempotent, $e_1$, which is not an identity element as $m^{12}(e_1,e_2)=0\neq e_2$, it follows that $A$ is not unital.
  \item If $A=(\langle e_1,e_2\rangle,m^{13})$: because 
  \begin{align*}
    m^{13}(m^{13}(e_1,e_1),e_2)&=e_2,\\
    m^{13}(e_1,m^{13}(e_1,e_2))&=e_2,\\
    m^{13}(m^{13}(e_2,e_2),e_1)&=0,\\
    m^{13}(e_2,m^{13}(e_2,e_1))&=0,  
  \end{align*}
  it is associative. While it is not unital as there is no idempotent. 
  \item If $A=(\langle e_1,e_2\rangle,m^{14})$: It can be calculated that 
  \begin{align*}
    m^{14}(m^{14}(e_1,e_1),e_2)&=0,\\
    m^{14}(e_1,m^{14}(e_1,e_2))&=0,\\
    m^{14}(m^{14}(e_2,e_2),e_1)&=0,\\
    m^{14}(e_2,m^{14}(e_2,e_1))&=0.  
  \end{align*}
  For the same reason above, this algebra is associative but not unital. 
  \item If $A=(\langle e_1,e_2\rangle,m^{15})$: Notice that 
  \begin{align*}
    m^{15}(m^{15}(e_1,e_1),e_2)&=-8e_1+8e_2,\\
    m^{15}(e_1,m^{15}(e_1,e_2))&=-6e_1+7e_2\neq -8e_1+8e_2, 
  \end{align*}
  hence this algebra is not associative, neither unital. 
  \item If $A=(\langle e_1,e_2\rangle,m^{16})$: It can be calculated that 
  \begin{align*}
    m^{16}(m^{16}(e_2,e_2),e_1)&=0,\\
    m^{16}(e_2,m^{16}(e_2,e_1))&=e_1\neq 0,   
  \end{align*}
  so that $A$ is not associative, neither unital. 
  \item If $A=(\langle e_1,e_2\rangle,m^{17})$: It is associative but not unital as the degeneration of multiplication.
  \item If $A=(\langle e_1,e_2\rangle,m^{8,1}_R)$: It can be calculated that 
  \begin{align*}
    m^{8,1}_R(m^{8,1}_R(e_1,e_1),e_2)&=\frac{1}{2}e_2,\\
    m^{8,1}_R(e_1,m^{8,1}_R(e_1,e_2))&=\frac{1}{4}e_2\neq  \frac{1}{2}e_2,
  \end{align*}
  so that this algebra is not associative, neither unital. 
  \item If $A=(\langle e_1,e_2\rangle,m^{8,2}_R)$: We have that 
  \begin{align*}
    m^{8,2}_R(m^{8,2}_R(e_1,e_1),e_2)&=\beta_2e_2,\\
    m^{8,2}_R(e_1,m^{8,2}_R(e_1,e_2))&=\beta_2^2e_2. 
  \end{align*}
  This shows that $A$ is associative only if $\beta_2=0$ or 1. If $\beta_2=0$, then $1-2\beta_2\in(\mathbb{K^*})^2$, which leads to a contradiction. If $\beta_2=1$, then $e_1$ is the identity element and hence this algebra is associative. And it should be noticed that $-1\notin (\mathbb{K^*})^2$ in this case. 
  \item If $A=(\langle e_1,e_2\rangle,m^{11}_R)$: It can be calculated that 
  \begin{align*}
    m^{11}_R(m^{11}_R(e_1,e_1),e_2)&=0,\\
    m^{11}_R(e_1,m^{11}_R(e_1,e_2))&=0,\\
    m^{11}_R(m^{11}_R(e_2,e_2),e_1)&=\lambda_2e_1,\\
    m^{11}_R(e_2,m^{11}_R(e_2,e_1))&=0.  
  \end{align*}
  This is to say, the algebra $A$ is associative if and only if $\lambda_2=0$. Same with the case of $m^{12}$, it is not unital as $e_1$, as the unique idempotent, is not an identity element. 
  \item If $A=(\langle e_1,e_2\rangle,m^{14,1}_R)$: We can calculate that 
  \begin{align*}
    m^{14,1}_R(m^{14,1}_R(e_1,e_1),e_2)&=0,\\
    m^{14,1}_R(e_1,m^{14,1}_R(e_1,e_2))&=\alpha_2e_1+(1+\alpha_2)e_2,\\
    m^{14,1}_R(m^{14,1}_R(e_2,e_2),e_1)&=0,\\
    m^{14,1}_R(e_2,m^{14,1}_R(e_2,e_1))&=\alpha_2^2e_1+\alpha_2e_2.  
  \end{align*}
  According to Lemma~\ref{Associative and unital lemma}, this algebra can not be associative, neither unital. 
  \item If $A=(\langle e_1,e_2\rangle,m^{14,2}_R)$: We have 
  \begin{align*}
    m^{14,2}_R(m^{14,2}_R(e_1,e_1),e_2)&=0,\\
    m^{14,2}_R(e_1,m^{14,2}_R(e_1,e_2))&=\alpha_2e_2,\\
    m^{14,2}_R(m^{14,2}_R(e_2,e_2),e_1)&=0,\\
    m^{14,2}_R(e_2,m^{14,2}_R(e_2,e_1))&=\alpha_2^2e_1.  
  \end{align*}
  So this algebra is associative if and only if $\alpha_2=0$. While this is conflict to $2\alpha_2+1\notin \mathbb{K}^2$, illustrating that $A$ is neither associative nor unital. 
  \item If $A=(\langle e_1,e_2\rangle,m^{15,1}_R)$: It is much more complicated that 
  \begin{align*}
    m^{15,1}_R(m^{15,1}_R(e_1,e_1),e_2)
    &=\alpha_4e_1+\beta_4e_2,\\
    m^{15,1}_R(e_1,m^{15,1}_R(e_1,e_2))
    &=\alpha_2\beta_2e_1+(\alpha_2+\beta_2^2)e_2,\\
    m^{15,1}_R(m^{15,1}_R(e_2,e_2),e_1)
    &=\alpha_2\beta_4e_1+(\alpha_4+\beta_2\beta_4)e_2,\\
    m^{15,1}_R(e_2,m^{15,1}_R(e_2,e_1))
    &=(\alpha_2^2+\alpha_4\beta_2)e_1+(\alpha_2+\beta_4)\beta_2e_2.
  \end{align*}
  And the algebra is associative if and only if 
  $\begin{cases}
    \alpha_4=\alpha_2\beta_2,\\
    \beta_4=\alpha_2+\beta_2^2,\\
    \alpha_2\beta_4=\alpha_2^2+\alpha_4\beta_2,\\
    \alpha_4+\beta_2\beta_4=(\alpha_2+\beta_4)\beta_2. \\
  \end{cases}$ 
  As the first and second equalities deduce the third and the fourth ones, the condition is equivalent to $\begin{cases}
    \alpha_4=\alpha_2\beta_2,\\
    \beta_4=\alpha_2+\beta_2^2.\\
  \end{cases}$ And in this case, the polynomial $P_A(y)$ can be written as $-1+y(5\alpha_2+\beta_2^2)+y^2(-8\alpha_2^2-2\alpha_2\beta_2^2)+y^3(4\alpha_2^3+\alpha_2^2\beta_2^2)$. Additionally, as this algebra has no idempotent by Remark~\ref{Idempotent remark}, it is not unital.
\end{enumerate}
\end{proof}

\subsubsection{Characteristic $2$ case}
In \cite{HUI2020}, the case of characteristic being $2$ without associative condition is well studied. The case of characteristic being $2$ with associative condition and coefficient field being algebraically closed is explored in \cite{JJJ2007}. However, what we need is a conclusion with general coefficient field. So we provide a classification here based on the method of Remm and Goze used in \cite{EM2017}.
\begin{theorem}\label{The classification with characteristic being 2}
  Let $\mathbb{K}$ be a field with characteristic $2$, $A$ be a commutative rank-$2$ associative $\mathbb{K}$-algebra. Then $A$ is isomorphic to one of follows, where $e_1$ and $e_2$ are the generators of the algebra. 
  \begin{itemize}
    \setlength{\itemsep}{5pt}
    \setlength{\parsep}{5pt}
    \setlength{\parskip}{5pt}
    \item \(\begin{cases}
      m_2^1(e_1,e_1)=e_1,\\
      m_2^1(e_1,e_2)=e_2,\\
      m_2^1(e_2,e_2)=e_2.\\
    \end{cases}
    \,
    \begin{cases}
      m_2^{2}(e_1,e_1)=e_1,\\
      m_2^{2}(e_1,e_2)=0,\\
      m_2^{2}(e_2,e_2)=0.\\
    \end{cases}
    \, 
    \begin{cases}
      m_2^{3}(e_1,e_1)=e_1,\\
      m_2^{3}(e_1,e_2)=0,\\
      m_2^{3}(e_2,e_2)=e_2.\\
    \end{cases}
    \, 
    \begin{cases}
      m_2^{4}(e_1,e_1)=e_1,\\
      m_2^{4}(e_1,e_2)=e_2,\\
      m_2^{4}(e_2,e_2)=\alpha_4e_1.\\
    \end{cases}\)
    \item \(\begin{cases}
      m_2^{5}(e_1,e_1)=e_1,\\
      m_2^{5}(e_1,e_2)=e_2,\\
      m_2^{5}(e_2,e_2)=\alpha_4e_1+e_2.\\
    \end{cases}
    \,\beta_4\neq 0,\,x^2+x+\alpha_4\text{ has no roots in }\mathbb{K}\setminus\{0,1\}.\)
    \item \(\begin{cases}
      m_2^6(e_1,e_1)=e_2,\\
      m_2^6(e_1,e_2)=0,\\
      m_2^6(e_2,e_2)=0.\\
    \end{cases}
    \quad 
    \begin{cases}
      m_{2}^7(e_1,e_1)=0,\\
      m_{2}^7(e_1,e_2)=0,\\
      m_{2}^7(e_2,e_2)=0.\\
    \end{cases}\)
  \end{itemize}
  If the coefficient field $\mathbb{K}$ is not algebraically closed, we have also the following algebra
  \[\begin{cases}
    m_{2,R}(e_1,e_1)=e_2,\\
    m_{2,R}(e_1,e_2)=\alpha_2e_1+\beta_2e_2,\\
    m_{2,R}(e_2,e_2)=\alpha_2\beta_2e_1+(\alpha_2+\beta_2^2)e_2.\\
  \end{cases}
  y^3\alpha_2\beta_2+y(\alpha_2+\beta_2^2)+1=0\text{ without roots}.\] 
  Besides, $m_2^1,m_2^3,m_2^4,m_2^5$ are unital while others are not.
\end{theorem}
\begin{proof}
Suppose $\mathbb{K}$ is a field with characteristic $2$, $A$ is a $2$-dimensional commutative associative $\mathbb{K}$-algebra. 
\begin{enumerate}
  \item If there are two independent idempotents, denoted by $e_1,e_2$. Then this algebra is generated by these two elements, with multiplication given by 
  \[\begin{cases}
    m_2(e_1,e_1)=e_1,\\
    m_2(e_1,e_2)=\alpha_2e_1+\beta_2e_2,\\
    m_2(e_2,e_2)=e_2.\\
  \end{cases}\]
  According to Lemma~\ref{Associative and unital lemma}, we know $A$ is associative if and only if 
  \[m_2(m_2(e_1,e_1),e_2)=m_2(e_1,m_2(e_1,e_2)),\]
  \[m_2(m_2(e_2,e_2),e_1)=m_2(e_2,m_2(e_2,e_1)).\]
  This is to say 
  \[\alpha_2^2=\alpha_2,\,\beta_2^2=\beta_2,\,\alpha_2\beta_2=0.\]
  Hence $(\alpha_2,\beta_2)$ can only be one of $(0,0),(0,1),(1,0)$, and for each case, the algebra has identity element $e_1+e_2, e_1$ or $e_2$ respectively. In addition, each of these three algebras is generated by an identity element and a nondegenerate idempotent, which implies that they are isomorphic to each other. Select a representative algebra in this isomorphic class as
  \[\begin{cases}
    m_2^1(e_1,e_1)=e_1,\\
    m_2^1(e_1,e_2)=e_2,\\
    m_2^1(e_2,e_2)=e_2.\\
  \end{cases}\]
  In this algebra, $e_1$ is the identity element and $e_2$ is the nondegenerate idempotent.
  \item Suppose there is only one idempotent, say $e_1$, we can represent the algebra by 
  \[\begin{cases}
    m_2(e_1,e_1)=e_1,\\
    m_2(e_1,e_2)=\alpha_2e_1+\beta_2e_2,\\
    m_2(e_2,e_2)=\alpha_4e_1+\beta_4e_2.\\
  \end{cases}\]
  Note that for each $k\in\mathbb{K}$
  \[m_2(e_1,(ke_1+e_2))=(\alpha_2+k)e_1+\beta_2e_2=(\alpha_2+k(\beta_2+1))e_1+\beta_2(ke_1+e_2),\]
  so that we replace $e_2$ by $ke_1+e_2$ and convert the coefficient $\alpha_2$ to $0$ if $\beta_2\neq 1$. 
  \begin{enumerate}
    \item If $\beta_2\neq 1$, then according to the analysis above, we can suppose $\alpha_2=0$ and the algebra becomes as 
    \[\begin{cases}
      m_2(e_1,e_1)=e_1,\\
      m_2(e_1,e_2)=\beta_2e_2,\\
      m_2(e_2,e_2)=\alpha_4e_1+\beta_4e_2.\\
    \end{cases}\]
    According to Lemma~\ref{Associative and unital lemma}, the associativity is equivalent to 
    \[m_2(m_2(e_1,e_1),e_2)=m_2(e_1,m_2(e_1,e_2)),\]
    \[m_2(m_2(e_2,e_2),e_1)=m_2(e_2,m_2(e_2,e_1)).\]
    The first equality gives out that $\beta_2^2=\beta_2$, so $\beta_2=0$ since $\beta_2\neq1$. On the other hand, the second equality implies that $\alpha_4=0$. Now the algebra becomes as 
    \[\begin{cases}
      m_2(e_1,e_1)=e_1,\\
      m_2(e_1,e_2)=0,\\
      m_2(e_2,e_2)=\beta_4e_2.\\
    \end{cases}\]
    Depends on whether $\beta_4=0$ or not, it is isomorphic to the first or second representative algebra as follows. 
    \[\begin{cases}
      m_2^{2}(e_1,e_1)=e_1,\\
      m_2^{2}(e_1,e_2)=0,\\
      m_2^{2}(e_2,e_2)=0.\\
    \end{cases}
    \quad 
    \begin{cases}
      m_2^{3}(e_1,e_1)=e_1,\\
      m_2^{3}(e_1,e_2)=0,\\
      m_2^{3}(e_2,e_2)=e_2.\\
    \end{cases}\]
    \item If $\beta_2=1$, then the algebra becomes as 
    \[\begin{cases}
      m_2(e_1,e_1)=e_1,\\
      m_2(e_1,e_2)=\alpha_2e_1+e_2,\\
      m_2(e_2,e_2)=\alpha_4e_1+\beta_4e_2.\\
    \end{cases}\]
    Then $m_2(m_2(e_1,e_1),e_2)=\alpha_2e_1+e_2$ and $m_2(e_1,m_2(e_1,e_2))=2\alpha_2e_1+e_2$. And the associativity of $A$ shows that $\alpha_2=2\alpha_2$, hence $\alpha_2=0$. In this way, $e_1$ is identity element and $A$ is associative by the second statement of Lemma~\ref{Associative and unital lemma}. 
    \par 
    For each $x,y\in\mathbb{K}$, the element $xe_1+ye_2$ is idempotent if and only if 
    \[xe_1+ye_2=x^2e_1+y^2(\alpha_4e_1+\beta_4e_2)=(x^2+y^2\alpha_4)e_1+y^2\beta_4e_2.\]
    As the only idempotent is $e_1$, the equation 
    $\begin{cases}
      x^2+y^2\alpha_4=x,\\
      y^2\beta_4=y.\\
    \end{cases}$
    should have no solutions other than 
    $\begin{cases}
      x=1,\\
      y=0.\\
    \end{cases}$ 
    and 
    $\begin{cases}
      x=0,\\
      y=0.\\
    \end{cases}$ When $y=0$, it is easy to found that $x$ can only be $0$ or $1$. When $y\neq 0$, we have $y=\beta_4^{-1}$ if $\beta_4\neq 0$, and the first equality can be rewritten as $x^2+x+\alpha_4\beta_4^{-2}=0$. And we can set the coefficient $\beta_4$ to be $1$ by replacing $e_2$ with $\beta_4^{-1}e_2$. While if $\beta_4=0$, it can be deduced that $y=0$ and $x^2=x$, which has no other solution indeed. In all, the algebras in this case is unital and isomorphic to
    \[\begin{cases}
      m_2^{4}(e_1,e_1)=e_1,\\
      m_2^{4}(e_1,e_2)=e_2,\\
      m_2^{4}(e_2,e_2)=\alpha_4e_1,\\
    \end{cases} \text{ or }\]
    \[\begin{cases}
      m_2^{5}(e_1,e_1)=e_1,\\
      m_2^{5}(e_1,e_2)=e_2,\\
      m_2^{5}(e_2,e_2)=\alpha_4e_1+e_2.\\
    \end{cases}
  x^2+x+\alpha_4=0\text{ has no roots in }\mathbb{K}\setminus\{0,1\}.\]
  \end{enumerate}   
  \item If there is no idempotent but exist an element in $A$ whose square is nonzero. Denote this element by $e_1$ and $m_2(e_1,e_1)$ by $e_2$. If $e_1$ and $e_2$ are not independent, say $e_2=ke_1$ with $k\in\mathbb{K^*}$, then $k^{-1}e_1$ is an idempotent, which is a contradiction. This means $e_1$ and $e_2$ form a basis of $A$ and we can represent the multiplication by 
  \[\begin{cases}
    m_2(e_1,e_1)=e_2,\\
    m_2(e_1,e_2)=\alpha_2e_1+\beta_2e_2,\\
    m_2(e_2,e_2)=\alpha_4e_1+\beta_4e_2.\\
  \end{cases}\]
  Then for each $x,y\in\mathbb{K}$, $xe_1+ye_2$ is idempotent if and only if 
  \[xe_1+ye_2=x^2e_2+y^2(\alpha_4e_1+\beta_4e_2)=y^2\alpha_4e_1+(x^2+y^2\beta_4)e_2,\]
  or equivalently, $\begin{cases}
    x=y^2\alpha_4,\\
    y^4\alpha_4^2+y^2\beta_4+y=0.\\
  \end{cases}$ As this algebra contains no idempotent, this equation has no solution rather than $x=y=0$, hence the equation 
  \begin{equation}
    y^3\alpha_4^2+y\beta_4+1=0
    \label{No idempotent 1}
  \end{equation}
has no solution in $\mathbb{K}$. 

Now let us turn to the associativity. It can be calculated directly that
  \begin{align*}
    m_2(m_2(e_1,e_1),e_2)
    &=\alpha_4e_1+\beta_4e_2,\\
    m_2(e_1,m_2(e_1,e_2))
    &=\alpha_2\beta_2e_1+(\alpha_2+\beta_2^2)e_2,\\
    m_2(m_2(e_2,e_2),e_1)
    &=\alpha_2\beta_4e_1+(\alpha_4+\beta_2\beta_4)e_2,\\
    m_2(e_2,m_2(e_2,e_1))
    &=(\alpha_2^2+\alpha_4\beta_2)e_1+(\alpha_2+\beta_4)\beta_2e_2.
  \end{align*}
  Then Lemma~\ref{Associative and unital lemma} shows that $\begin{cases}
    \alpha_4=\alpha_2\beta_2,\\
    \beta_4=\alpha_2+\beta_2^2,\\
    \alpha_2\beta_4=\alpha_2^2+\alpha_4\beta_2,\\
    \alpha_4+\beta_2\beta_4=(\alpha_2+\beta_4)\beta_2.\\
  \end{cases}$ Note that the third and fourth equalities always hold if the first and second do, hence this algebra is associative if and only if $\begin{cases}
    \alpha_4=\alpha_2\beta_2,\\
    \beta_4=\alpha_2+\beta_2^2.\\
  \end{cases}$ Then the multiplication can be written as 
  \[\begin{cases}
    m_2(e_1,e_1)=e_2,\\
    m_2(e_1,e_2)=\alpha_2e_1+\beta_2e_2,\\
    m_2(e_2,e_2)=\alpha_2\beta_2e_1+(\alpha_2+\beta_2^2)e_2.\\
  \end{cases}\]
  And the Equation~\ref{No idempotent 1} has no solution in $\mathbb{K}$ can be translated as 
  \[y^3\alpha_2^2\beta_2^2+y(\alpha_2+\beta_2^2)+1=0.\]

  When $\mathbb{K}$ is algebraically closed, this equation has no solution if and only if its all coefficients, $\alpha_2^2\beta_2^2$ and $\alpha_2+\beta_2^2$ vanish, which is equivalent to $\alpha_2=\beta_2=0.$ In this case, the algebra can be written as 
  \[\begin{cases}
    m_2^6(e_1,e_1)=e_2,\\
    m_2^6(e_1,e_2)=0,\\
    m_2^6(e_2,e_2)=0.\\
  \end{cases}\]
  If $\mathbb{K}$ is not algebraically closed, we have to stop at 
  \[\begin{cases}
    m_{2,R}(e_1,e_1)=e_2,\\
    m_{2,R}(e_1,e_2)=\alpha_2e_1+\beta_2e_2,\\
    m_{2,R}(e_2,e_2)=\alpha_2\beta_2e_1+(\alpha_2+\beta_2^2)e_2,\\
  \end{cases}
  y^3\alpha_2^2\beta_2^2+y(\alpha_2+\beta_2^2)+1=0\text{ without roots}.\]
  \item If $m_2(v, v)=0$ for any $v\in A$, then we can represent the algebra by 
  \[\begin{cases}
    m_{2}(e_1,e_1)=0,\\
    m_{2}(e_1,e_2)=\alpha_2e_1+\beta_2e_2,\\
    m_{2}(e_2,e_2)=0.\\
  \end{cases}\]
  While we obtain $\alpha_2=\beta_2=0$ from the associativity. So the last kind of algebra degenerates as 
  \[\begin{cases}
    m_{2}^7(e_1,e_1)=0,\\
    m_{2}^7(e_1,e_2)=0,\\
    m_{2}^7(e_2,e_2)=0.\\
  \end{cases}\]
\end{enumerate}
\end{proof}

Note that the polynomial that has no roots for $m_{2,R}$ is nothing but the the polynomial $P_A(y)$ modulated by $2$. Back to the case with general field, combining Theorem \ref{Theorem in EM2017}, Proposition \ref{The further proposition with characteristic other than 2} and Theorem \ref{The classification with characteristic being 2}, we obtain the following corollary.
\begin{corollary}
Suppose $\mathbb{K}$ is a field, $A$ is a commutative 2-dimensional associative $\mathbb{K}$-algebra and the multiplication $m: A\otimes A\to A$ is surjective, then $A$ is unital except being isomorphic to $A_R$, where the multiplication $m_R$ is given by 
  \begin{align*}
    m=m_R:&A\otimes A\to A,\\
    &e_1\otimes e_1\mapsto e_2,\\
    &e_1\otimes e_2\mapsto \alpha_2e_1+\beta_2e_2,\\
    &e_2\otimes e_2\mapsto \alpha_2\beta_2e_1+(\alpha_2+\beta_2^2)e_2, 
  \end{align*}
  here the coefficients $\alpha_2,\beta_2\in\mathbb{K}$ make the polynomial
  \[P_R(y)=-1+y(5\alpha_2+\beta_2^2)+y^2(-8\alpha_2^2-2\alpha_2\beta_2^2)+y^3(4\alpha_2^3+\alpha_2^2\beta_2^2)\]
  has no roots in $\mathbb{K}$. 
\end{corollary}
Now we make some further analysis to the polynomial $P_R(y)$. If the coefficient $\alpha_2=0$, then the multiplication is not surjective. If $\alpha_2\neq 0$, then $y=\alpha_2^{-1}$ is a solution of $P_A(y)=0$. So we obtain the following theorem, which can be regarded as Theorem \ref{theorem1.2} with field coefficients.

\begin{theorem}
Any rank-$2$ commutative associative algebra with field coefficients and surjective multiplication is unital.
\label{Unital theorem}
\end{theorem}

\begin{remark}
The subjectivity of multiplication is necessary for the existence of identity element. While we do not know whether the rank-$2$ condition can be extended to higher ranks.
\end{remark}

\subsection{From field to PID}
In this subsection, first we extend Theorem~\ref{Unital theorem} to the case with PID coefficients and complete the proof of Theorem \ref{theorem1.2}.
\begin{proof}
  Suppose $A$ is a rank-$2$ $R$-algebra generated by $e_1$ and $e_2$ with surjective multiplication $m$ and $R$ is a principal ideal domain. Denote the quotient field of $R$ by $\overline{R}$, and the algebra generated by $e_1$ and $e_2$ under $\overline{R}$ by $\overline{A}$. The multiplication on $A$ can be extended to $\overline{A}$ naturally and denoted by $\overline{m}$, which is also commutative, associative and surjective as a linear map from $\overline{A}\otimes \overline{A}$ to $\overline{A}$. Now $(\overline{A},\overline{m})$ is a rank-$2$ associative algebra with field coefficient and surjective multiplication, hence it is unital by Theorem~\ref{Unital theorem}. Denote this identity element by $1_{\overline{A}}$ and we know that there exist $r_{11},r_{21}\in R$ and $r_{12},r_{22}\in R^*$ such that $1_{\overline{A}}=\frac{r_{11}}{r_{12}}e_1+\frac{r_{21}}{r_{22}}e_2$. Without the loss of generality, we may suppose 
  \[a\cdot 1_{\overline{A}}=r_1e_1+r_2e_2,\,(a,r_1,r_2)\text{ are prime in }R. \]
  As before, we still represent the multiplication by 
  \[\begin{cases}
    e_1\cdot e_1=\alpha_1e_1+\beta_1e_2,\\
    e_1\cdot e_2=\alpha_2e_1+\beta_2e_2,\\
    e_2\cdot e_2=\alpha_4e_1+\beta_4e_2,\\
  \end{cases}\]
  where we use the notation $\cdot$ as there is no need to take apart different kinds of multiplications. 

As $1_{\overline{A}}$ is the identity element of $\overline{A}$, we have that 
  \[ae_1=(r_1e_1+r_2e_2)\cdot e_1=(r_1\alpha_1+r_2\alpha_2)e_1+(r_1\beta_1+r_2\beta_2)e_2,\]
  \[ae_2=(r_1e_1+r_2e_2)\cdot e_2=(r_1\alpha_2+r_2\alpha_4)e_1+(r_1\beta_2+r_2\beta_4)e_2.\]
This is to say, 
  \[\begin{cases}
    r_1\alpha_2+r_2\alpha_4=0,\\
    r_1\beta_1+r_2\beta_2=0,\\
    r_1(\alpha_1-\beta_2)+r_2(\alpha_2-\beta_4)=0,\\
    a=r_1\alpha_1+r_2\alpha_2.\\
  \end{cases}\]
Suppose $(r_1, r_2)r_1'=r_1$ and $(r_1, r_2)r_2'=r_2$ then there exist $u, v, t\in R$ such that 
  \[\alpha_2=ur_2',\quad \alpha_4=-ur_1',\]
  \[\beta_1=vr_2',\quad \beta_2=-vr_1',\]
  \[\alpha_1-\beta_2=tr_2',\quad \alpha_2-\beta_4=-tr_1'.\]
Therefore, we obtain 
  \begin{align*}
    a&=r_1\alpha_1+r_2\alpha_2\\
    &=(r_1,r_2)(r_1'(\beta_2+tr_2')+r_2'\alpha_2)\\
    &=(r_1,r_2)(r_1'(-vr_1'+tr_2')+r_2'\alpha_2)\\
    &=(r_1,r_2)(-vr_1'^2+tr_1'r_2'+ur_2'^2).
  \end{align*}
Combine this with $(a, r_1, r_2)=1$ in $R$, we obtain that $(r_1, r_2)=1$ and $r_{1, 2}'=r_{1, 2}$. 

As $R$ is a principal ideal domain, B\'{e}zout's lemma tells us that there exist $p, q\in R$, such that $qr_1-pr_2=1$, which guarantees that $e_1$ and $e_2$ can be obtained by $R$-linear combinations of $a\cdot1_{\overline{A}}=r_1e_1+r_2e_2$ and $pe_1+qe_2$. This is to say, 
  \[A=\langle u_a,x\rangle,\text{ where }u_a=a\cdot 1_{\overline{A}},\,x=pe_1+qe_2. \]
  And the multiplication, respect to $u_a$ and $x$ could be represented by 
  \[\begin{cases}
    u_a\cdot u_a=au_a,\\
    u_a\cdot x=ax,\\
    x\cdot x=hx+lu_a.\\
  \end{cases}h,l\in R.\]
  Then it is surjective if and only if there exist $b, c, d, e, f, g\in R$ such that 
  \[\begin{pmatrix}
    b&c&d\\
    e&f&g
  \end{pmatrix}\begin{pmatrix}
    a&0\\
    0&a\\
    l&h
  \end{pmatrix}=\begin{pmatrix}
    1&0\\
    0&1
  \end{pmatrix}.\]
  From $fa+gh=1$, we have $(a,h)=1$. Combine this with $ca+dh=0$, we know there exists $m\in R$ such that $c=mh$ and $d=-ma$. Notice that $ba+dl=1$, it can be finally deduced that
  \[a(b-ml)=1. \]
  So that $a$ is invertible in $R$ and $1_{\overline{A}}\in A$, hence $A$ is unital as desired. 
\end{proof}

Now we turn to the proof of Theorem \ref{theorem1.1}
\begin{proof}
Compare the definition of nearly Frobenius algebra and that of Frobenius algebra, we find that the only difference is the existence of the unit and counit. Now Theorem \ref{theorem1.2} suggests that for a 2-dimensional commutative associative algebra over a PID, if the multiplication is surjective, then it is unital. For the counit, denote the dual map of the comultiplication $\Delta: A\to A\otimes A$ by $\Delta^*: A^*\otimes A^*\to A^*$. Since $\Delta$ is injective, then $\Delta^*$ is surjective. Again, according to Theorem \ref{theorem1.2}, we find that $(A^*, \Delta^*)$ is unital. It follows that $(A, \Delta)$ has a counit.
\end{proof}

\begin{remark}
We remark that very recently, Asrorov, Bekbaev and Rakhimov classified all possible rank-2 Frobenius algebra structures over a basic field. The reader is referred to \cite{ABR2023} for more details.
\end{remark}

\section{Noncommutative case}\label{section4}
It is natural to ask what if we delete the condition of commutativity, whether the existence of the identity element is still valid. The following theorem suggests that there does exist an element $e_2$ which acts like an identity respect to partial multiplication.

\begin{theorem}
If (A, m) is a rank-$2$ algebra over a field $\mathbb{K}$, with associative, surjective, noncommutative multiplication $m: A\otimes A\to A$, then $A$ is spanned by $e_1$ and $e_2$ such that the multiplication structure $m$ is given by
  $$\begin{cases}
    e_1\otimes e_1\mapsto 0,\\
    e_1\otimes e_2\mapsto 0,\\
    e_2\otimes e_1\mapsto e_1,\\
    e_2\otimes e_2\mapsto e_2;
  \end{cases}\quad \text{or}\quad \begin{cases}
    e_1\otimes e_1\mapsto 0,\\
    e_1\otimes e_2\mapsto e_1,\\
    e_2\otimes e_1\mapsto 0,\\
    e_2\otimes e_2\mapsto e_2. 
  \end{cases}$$
\end{theorem}
\begin{proof}
If $\mathrm{char}\,\mathbb{K}\neq 2$, then the multiplication can be decomposed as the sum of the skew-symmetric part $m_a=\frac{1}{2}(m-m\circ \mathrm{Perm})\neq 0$ and the symmetric part $m_s=\frac{1}{2}(m+m\circ\mathrm{Perm})$. 

We claim that there exist $e_1$ and $e_2$ generate $A$ such that $m_a(e_1,e_2)=e_1$. In fact, because $m_a\neq 0$, there must exist $e_1'$ and $e_2'$ in $A$ such that $m_a(e_1',e_2')\neq 0$. As $m_a$ is skew-symmetric, $e_1'$ and $e_2'$ are linear independent hence generate $A$. Without loss of generality, suppose $m_a(e_1',e_2')=xe_1'+ye_2'$, where $x\neq 0$, then
  $$m_a(e_1'+x^{-1}ye_2',x^{-1}e_2')=e_1'+x^{-1}ye_2'.$$
So that $e_1'+x^{-1}ye_2'$ and $x^{-1}e_2'$ are exactly what we want. 

Suppose $m_s: e_1\otimes e_1\mapsto \alpha_1e_1+\beta_1e_2,\,e_1\otimes e_2\mapsto \alpha_2e_1+\beta_2e_2,\,e_2\otimes e_2\mapsto \alpha_4e_1+\beta_4e_2,$ then we have 
  \begin{align*}
    m=m_s+m_a:& e_1\otimes e_1\mapsto \alpha_1e_1+\beta_1e_2,\\
    &e_1\otimes e_2\mapsto (\alpha_2+1)e_1+\beta_2e_2,\\
    &e_2\otimes e_1\mapsto (\alpha_2-1)e_1+\beta_2e_2,\\
    &e_2\otimes e_2\mapsto \alpha_4e_1+\beta_4e_2. 
  \end{align*}

  Now we analyze the associativity of this algebra and we write $m(a,b)$ as $a\cdot b$.
  \begin{flalign*}
    0=&(e_1\cdot e_1)\cdot e_1-e_1\cdot(e_1\cdot e_1)&\\
    =&(\alpha_1e_1+\beta_1e_2)\cdot e_1-e_1\cdot (\alpha_1e_1+\beta_1e_2)&\\
    =&\beta_1(e_2\cdot e_1-e_1\cdot e_2)&\\
    =&-2\beta_1e_1\Longrightarrow\beta_1=0&\\
  \end{flalign*}
 \begin{flalign*}
    0=&(e_1\cdot e_1)\cdot e_2-e_1\cdot(e_1\cdot e_2)&\\
    =&\alpha_1e_1\cdot e_2-e_1\cdot((\alpha_2+1)e_1+\beta_2e_2)&\\
    =&(\alpha_1-\beta_2)e_1\cdot e_2-(\alpha_2+1)e_1\cdot e_1&\\
    =&(\alpha_1-\beta_2)((\alpha_2+1)e_1+\beta_2e_2)-(\alpha_2+1)\alpha_1 e_1&\\
    =&-(\alpha_2+1)\beta_2e_1+(\alpha_1-\beta_2)\beta_2e_2\Longrightarrow(\alpha_2+1)\beta_2=(\alpha_1-\beta_2)\beta_2=0&\\
 \end{flalign*}
 \begin{flalign*}
    0=&(e_2\cdot e_1)\cdot e_1-e_2\cdot(e_1\cdot e_1)&\\
    =&((\alpha_2-1)e_1+\beta_2e_2)\cdot e_1-e_2\cdot \alpha_1e_1&\\
    =&(\beta_2-\alpha_1)e_2\cdot e_1+(\alpha_2-1)e_1\cdot e_1&\\
    =&(\beta_2-\alpha_1)((\alpha_2-1)e_1+\beta_2e_2)+(\alpha_2-1)\alpha_1 e_1&\\
    =&(\alpha_2-1)\beta_2e_1+(\beta_2-\alpha_1)\beta_2e_2\Longrightarrow(\alpha_2-1)\beta_2=(\beta_2-\alpha_1)\beta_2=0&\\
 \end{flalign*}
Now we conclude that $\beta_2=0$ from $(\alpha_2+1)\beta_2=(\alpha_2-1)\beta_2=0$. For the fourth case, we have
 \begin{flalign*}
    0=&(e_1\cdot e_2)\cdot e_1-e_1\cdot(e_2\cdot e_1)&\\
    =&(\alpha_2+1)e_1\cdot e_1-e_1\cdot (\alpha_2-1)e_1&\\
    =&2e_1\cdot e_1\Longrightarrow \alpha_1=0&\\
  \end{flalign*}
  Then the multiplication can be written as 
  \begin{align*}
    m=m_s+m_a: & e_1\otimes e_1\mapsto 0,\\
    &e_1\otimes e_2\mapsto (\alpha_2+1)e_1,\\
    &e_2\otimes e_1\mapsto (\alpha_2-1)e_1,\\
    &e_2\otimes e_2\mapsto \alpha_4e_1+\beta_4e_2. 
  \end{align*}
  We continue to go on analysis:
  \begin{flalign*}
    0=&(e_2\cdot e_2)\cdot e_1-e_2\cdot(e_2\cdot e_1)&\\
    =&(\alpha_4e_1+\beta_4e_2)\cdot e_1-e_2\cdot (\alpha_2-1)e_1&\\
    =&-(\alpha_2-\beta_4-1)(\alpha_2-1)e_1\Longrightarrow(\alpha_2-\beta_4-1)(\alpha_2-1)=0&\\
    0=&(e_1\cdot e_2)\cdot e_2-e_1\cdot(e_2\cdot e_2)&\\
    =&(\alpha_2+1)e_1\cdot e_2-e_1\cdot (\alpha_4e_1+\beta_4e_2)&\\
    =&(\alpha_2-\beta_4+1)(\alpha_2+1)e_1\Longrightarrow(\alpha_2-\beta_4+1)(\alpha_2+1)=0&
 \end{flalign*}
Since $(\alpha_2-\beta_4-1)(\alpha_2-1)=(\alpha_2-\beta_4+1)(\alpha_2+1)=0$, then $(\alpha_2,\beta_4)\in\{(1, 2), (-1, -2)\}$.
 \begin{flalign*}
    0=&(e_2\cdot e_1)\cdot e_2-e_2\cdot(e_1\cdot e_2)&\\
    =&(\alpha_2-1)e_1\cdot e_2-(\alpha_2+1)e_2\cdot e_1&\\
    =&0&\\
    0=&(e_2\cdot e_2)\cdot e_2-e_2\cdot(e_2\cdot e_2)&\\
    =&(\alpha_4e_1+\beta_4e_2)\cdot e_2-e_2\cdot(\alpha_4e_1+\beta_4e_2)&\\
    =&\alpha_4(e_1\cdot e_2-e_2\cdot e_1)&\\
    =&2\alpha_4e_1\Longrightarrow\alpha_4=0&
  \end{flalign*}
  So that the multiplication is given by 
  $$\begin{cases}
    e_1\cdot e_1=0,\\
    e_1\cdot e_2=2e_1,\\
    e_2\cdot e_1=0,\\
    e_2\cdot e_2=2e_2;\\
  \end{cases}\text{ or }\begin{cases}
    e_1\cdot e_1=0,\\
    e_1\cdot e_2=0,\\
    e_2\cdot e_1=-2e_1,\\
    e_2\cdot e_2=-2e_2.\\
  \end{cases}$$
  And they are isomorphic to the two algebraic structures stated in the theorem.
  
If $\mathrm{char}\,\mathbb{K}=2$, suppose that $\begin{cases}
    e_1\cdot e_1=\alpha_1e_1+\beta_1e_2,\\
    e_1\cdot e_2=\alpha_2e_1+\beta_2e_2,\\
    e_2\cdot e_1=\alpha_3e_1+\beta_3e_2,\\
    e_2\cdot e_2=\alpha_4e_1+\beta_4e_2;\\
  \end{cases}$
  the associativity implies that 
  \begin{align*}
    0=&(e_1\cdot e_1)\cdot e_1+e_1\cdot(e_1\cdot e_1)
    =(\alpha_1e_1+\beta_1e_2)\cdot e_1+e_1\cdot(\alpha_1e_1+\beta_1e_2)\\
    =&2\alpha_1e_1\cdot e_1+\beta_1(e_1\cdot e_2+e_2\cdot e_1)=(\alpha_2+\alpha_3)\beta_1e_1+\beta_1(\beta_2+\beta_3)e_2. 
  \end{align*}
  This shows that if $\beta_1\neq 0$, then $\alpha_2=\alpha_3$ and $\beta_2=\beta_3$, which is conflict to the noncommutative assumption. Hence it deduces that $\beta_1=0$. In the same way, we obtain $\alpha_4=0$ from $0=(e_2\cdot e_2)\cdot e_2+e_2\cdot(e_2\cdot e_2)$. In addition, based on 
  \begin{align*}
    0=&(e_1\cdot e_2)\cdot e_1+e_1\cdot(e_2\cdot e_1)
    =(\alpha_2e_1+\beta_2e_2)\cdot e_1+e_1\cdot(\alpha_3e_1+\beta_3e_2)\\
    =&(\alpha_2+\alpha_3)\alpha_1e_1+\beta_2(\alpha_3e_1+\beta_3e_2)+\beta_3(\alpha_2e_1+\beta_2e_2)\\
    =&((\alpha_1+\beta_2)\alpha_3+(\alpha_1+\beta_3)\alpha_2)e_1,\\
    0=&(e_1\cdot e_1)\cdot e_2+e_1\cdot(e_1\cdot e_2)
    =\alpha_1e_1\cdot e_1+e_1\cdot(\alpha_2e_1+\beta_2e_2)\\
    =&(\alpha_1+\beta_2)(\alpha_2e_1+\beta_2e_2)+\alpha_1\alpha_2e_1=\alpha_2\beta_2e_1+(\alpha_1+\beta_2)\beta_2e_2,\\
    0=&(e_2\cdot e_1)\cdot e_1+e_2\cdot(e_1\cdot e_1)
    =(\alpha_3e_1+\beta_3e_2)\cdot e_1+e_2\cdot \alpha_1e_1\\
    =&(\alpha_1+\beta_3)(\alpha_3e_1+\beta_3e_2)+\alpha_1\alpha_3e_1=\alpha_3\beta_3e_1+(\alpha_1+\beta_3)\beta_3e_2,
  \end{align*}
  we have $(\alpha_1+\beta_2)\alpha_3+(\alpha_1+\beta_3)\alpha_2=0$, $\begin{cases}
    \alpha_2\beta_2=0,\\
    (\alpha_1+\beta_2)\beta_2=0,\\
  \end{cases}$ $\begin{cases}
    \alpha_3\beta_3=0,\\
    (\alpha_1+\beta_3)\beta_3=0.\\
  \end{cases}$ Exchange the position of $e_1$ and $e_2$, and the left conditions of associativity tell us that $(\beta_4+\alpha_3)\beta_2+(\beta_4+\alpha_2)\beta_3=0$, $(\beta_4+\alpha_3)\alpha_3=0$ and $(\beta_4+\alpha_2)\alpha_2=0$. In all the associativity conditions are translated into following equations
  \begin{equation}
    \label{AssCondition}
    \begin{cases}
      \beta_1=\alpha_4=0,\\
      \alpha_2\beta_2=0,\\
      \alpha_3\beta_3=0,\\
      (\alpha_1+\beta_2)\beta_2=0,\\
      (\alpha_1+\beta_3)\beta_3=0,\\
      \alpha_2(\alpha_2+\beta_4)=0,\\
      \alpha_3(\alpha_3+\beta_4)=0,\\
      (\alpha_1+\beta_2+\beta_3)(\alpha_2+\alpha_3)=0,\\
      (\alpha_2+\alpha_3+\beta_4)(\beta_2+\beta_3)=0,\\
    \end{cases}
  \end{equation}
where the penultimate one equation is deduced by adding $\alpha_2\beta_2+\alpha_3\beta_3=0$ to $(\alpha_1+\beta_2)\alpha_3+(\alpha_1+\beta_3)\alpha_2=0$ and the last one equation can be obtained similarly.
\begin{itemize}
    \item If $\alpha_2=\alpha_3=0$, Eq.\ref{AssCondition} $\Longrightarrow\begin{cases}
      (\alpha_1+\beta_2)\beta_2=0,\\
      (\alpha_1+\beta_3)\beta_3=0,\\
      \beta_4(\beta_2+\beta_3)=0.\\
    \end{cases}$
    If $\beta_2+\beta_3=0$, then $e_1\cdot e_2=e_2\cdot e_1$, which is conflict to the noncommutative assumption. So that $\beta_2+\beta_3\neq 0$ and we have 
    $\beta_4=0,\, (\beta_2,\beta_3)=(0,\alpha_1)\text{ or }(\alpha_1,0),\,\alpha_1\neq 0.$
    Hence the multiplication is given by 
    $$\begin{cases}
      e_1\cdot e_1=\alpha_1e_1,\\
      e_1\cdot e_2=0,\\
      e_2\cdot e_1=\alpha_1e_2,\\
      e_2\cdot e_2=0;\\
    \end{cases}\text{ or }\begin{cases}
      e_1\cdot e_1=\alpha_1e_1,\\
      e_1\cdot e_2=\alpha_1e_2,\\
      e_2\cdot e_1=0,\\
      e_2\cdot e_2=0.\\
    \end{cases}$$
    Use $\alpha_1^{-1}e_1$ and $e_2$ to substitute $e_2$ and $e_1$ then we obtain the multiplication structure in the theorem. 
    \item If $\alpha_2=0,\,\alpha_3\neq 0$, Eq.\ref{AssCondition} $\Longrightarrow\begin{cases}
      \beta_3=0,\\
      \alpha_3(\alpha_3+\beta_4)=0,\\
      (\alpha_1+\beta_2)\alpha_3=0,\\
    \end{cases}\Longrightarrow\begin{cases}
      \alpha_3=\beta_4,\\
      \beta_2=\alpha_1.\\
    \end{cases}$
    And the multiplication structure becomes $\begin{cases}
      e_1\cdot e_1=\alpha_1e_1,\\
      e_1\cdot e_2=\alpha_1e_2,\\
      e_2\cdot e_1=\beta_4e_1,\\
      e_2\cdot e_2=\beta_4e_2.\\
    \end{cases}$

    If $\alpha_1=\beta_4=0$, then the multiplication is a zero map, conflicting to the subjectivity assumption. 

    If only one of $\alpha_1$ and $\beta_4$ equals to zero, similar to the former case, it is isomorphic to the multiplication structure mentioned in the theorem under base change. 

    If $\alpha_1\neq 0,\, \beta_4\neq 0$, then we may use $\alpha_1^{-1}e_1$ and $\alpha_1^{-1}e_1-\beta_4^{-1}e_2$ to substitute $e_2$ and $e_1$. 
    \item If $\alpha_2\neq 0,\,\alpha_3=0$, the conclusion holds symmetrically for the similar reason in the former case. 
    \item If $\alpha_2\neq 0,\,\alpha_3\neq 0$, Eq.\ref{AssCondition} $\Longrightarrow\begin{cases}
      \beta_2=\beta_3=0,\\
      \alpha_3(\alpha_3+\beta_4)=0,\\
      \alpha_2(\alpha_2+\beta_4)=0,\\
      \alpha_1(\alpha_2+\alpha_3)=0.\\
    \end{cases}$
    As $m$ is noncommutative, $\alpha_2\neq \alpha_3$, which furtherly shows that 
    $\alpha_1=0,\,(\alpha_2,\alpha_3)=(0,\beta_4)\text{ or }(\beta_4,0),\,\beta_4\neq 0.$
    And the corresponding multiplication structure is given by $\begin{cases}
      e_1\cdot e_1=0,\\
      e_1\cdot e_2=0,\\
      e_2\cdot e_1=\beta_4e_1,\\
      e_2\cdot e_2=\beta_4e_2.\\
    \end{cases}$ or $\begin{cases}
      e_1\cdot e_1=0,\\
      e_1\cdot e_2=\beta_4e_1,\\
      e_2\cdot e_1=0,\\
      e_2\cdot e_2=\beta_4e_2,\\
    \end{cases}$ which is isomorphic to the ones occurs before after a scalar base change by $\beta_4^{-1}$. 
  \end{itemize}
\end{proof}

\section{Link Homology induced by almost 2D-TQFT}\label{Homology induced by almost 2D-TQFT}
In \cite{Kho2000}, Khovanov constructed a bi-graded homology which categorifies the Jones polynomial. This homology is obtained by constructing a bi-graded chain complex to each link diagram and calculating its homology groups. This section focuses on the following question: which kind of almost 2D-TQFT can be used to define a link homology?
\subsection{The chain complex of a link diagram}
Given a link diagram $D$ with crossings points $\{c_1, c_2, \cdots, c_n\}$, we can resolve each crossing by $0$-resolution or $1$-resolution, see Figure~\ref{fig:Resolutions}.
\begin{figure}[htbp]
  \centering 
  \begin{tikzpicture}[scale=1,use Hobby shortcut]
    \begin{knot}[clip width=3.7,]
    \strand (0.5,0.5)..(0,0)..(-0.5,-0.5);
    \strand (-0.5,0.5)..(0,0)..(0.5,-0.5);
    \end{knot}
    \draw[->] (1,0)--(3,0);
    \node at (2,0.3) {$0$-resolution};
    \draw[->] (-1,0)--(-3,0);
    \node at (-2,0.3) {$1$-resolution};
    \draw (3.5,-0.5)..++(0.01,0.01)..(3.9,0)..(3.5,0.5)..++(-0.01,0.01) (4.5,-0.5)..++(-0.01,0.01)..(4.1,0)..(4.5,0.5)..++(0.01,0.01);
    \draw (-3.5,-0.5)..++(-0.01,0.01)..(-4,-0.1)..(-4.5,-0.5)..++(-0.01,-0.01) (-3.5,0.5)..++(-0.01,-0.01)..(-4,0.1)..(-4.5,0.5)..++(-0.01,0.01); 
  \end{tikzpicture}
  \caption{$0$-resolution and $1$-resolution}
  \label{fig:Resolutions}
\end{figure}
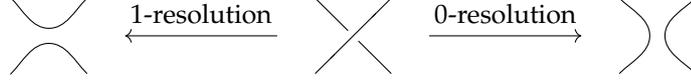

Now a state $s\in\{0, 1\}^n$ corresponds to a disjoint union of circles which can be regarded as an object in $\textbf{N2Cob}$. Let us denote it by $\Gamma_s$ and use $|s|$ to denote the number of 1's in $s$. Different states form a poset via $s_1<s_2$ if and only if all $1$-resolutions in $s_1$ are also $1$-resolutions in $s_2$. Each states can be seen as a vertex of a $n$-cube, and for each $s_1<s_2$, $|s_2|=|s_1|+1$, we assign to this edge a morphism of $\textbf{N2Cob}$ by the cobordism that is a saddle surface connecting the $0$-resolution in $s_1$ to the corresponding $1$-resolution in $s_2$, and a cylinder for other irrelevant circles. Denote this morphism by $W_{s_1}^{s_2}$. Then the cube $Cube(D)$ consisted by $s$'s and $W$'s is commutative by the property of $\textbf{N2Cob}$.

For two states $\{s_1, s_2\}$ with $s_1<s_2$ and $|s_2|=|s_1|+1$, define $\langle s_1, s_2\rangle$ to be the number of $1$-resolutions before the $0$-resolution in $s_1$ which is changed to a $1$-resolution in $s_2$. Now we can define the chain complex of the diagram $D$ as follows.
\begin{definition}
Suppose $D$ is a link diagram, its chain complex induced by an almost 2D-TQFT $\mathcal{F}^a$ is $C(D;\mathcal{F}^a)=(C^i(D;\mathcal{F}^a),d^i(D;\mathcal{F}^a))$, where the chain modules and differentials are given by 
  \[C^i(D;\mathcal{F}^a)=\mathop\oplus_{|s|=i}\mathcal{F}^a(\Gamma_s(D)),\quad d^i(D;\mathcal{F}^a)=\sum_{|s_1|=i}(-1)^{\langle s_1,s_2\rangle}\mathcal{F}^a(W_{s_1}^{s_2}).\]
  And the homology of this chain complex is 
  \[H^i(D;\mathcal{F}^a)=\ker d^i/\mathrm{Im}d^{i-1}.\]
\end{definition}

It can be checked that $d^{i+1}\circ d^i=0$ and different orders of the crossing points provide the same homology groups. Since different link diagrams represent the same link if and only if they are related by some Reidemeister moves. It follows that if for an almost 2D-TQFTs, the homology induced by it is preserved by Reidemeister moves, then its homology defines a link invariant.
 
\subsection{A review of Khovanov's universal theory}
In \cite{Kho2006}, Khovanov considered the question when a TQFT induces a homology which is invariant under the first Reidemeister move. He found such a TQFT, or the corresponding Frobenius algebra is of rank-2, hence the algebra $A$ is generated by the identity $1$ and another element $x$. In order to make this algebra satisfies all the conditions of a Frobenius algebra, he proved that it must be born from a universal Frobenius algebra $$A_4=R_4[x]/(x^2-hx-t),$$ where the coefficient ring 
$$R_4=\mathbb{Z}[a,c,e,f,h,t](ae-cf,af+chf-cet-1),$$
and the comultiplication and counit are defined as
$$\Delta(1)=(et-hf)1\otimes 1+ex\otimes x+f(1\otimes x+x\otimes 1),$$
$$\Delta(x)=ft1\otimes 1+et(1\otimes x+x\otimes 1)+(f+eh)x\otimes x,$$
$$\epsilon(1)=-c,\quad \epsilon(x)=a.$$ 
It is universal in the sense that for any rank-$2$ Frobenius algebra $A'$ with splitting element $x'$, there exists a unique homomorphism $\psi:A_4\to A'$ with $\psi(x)=x'$ that realizes $A'$ as a base change of $A_4$. While this is not the final algebra. Before giving the final one, we need to introduce the notion of \emph{twisting}.

Given a Frobenius algebra $A$ and an invertible element $y\in A$, we can twist the counit and comultiplication as $\epsilon'(x)=\epsilon(m(y,x))$, and $\Delta'(x)=\Delta(m(y^{-1},x))$ to obtain a new Frobenius algebra $A'=(A, \epsilon')$ whose unit and multiplication are preserved. Additionally, Khovanov proved that a twisting does not change the link homology induced by the corresponding TQFT \cite[Proposition 3]{Kho2006}.

\begin{remark}
Twisting by invertible elements is the only way to modify the counit and comultiplication in Frobenius algebras \cite[Theorem 1.6]{Ka1999}. 
\end{remark}

An important point is that the element $f+ex\in A_4$ is invertible with inverse $a+ch-cx$. Apply a twisting by this element to $A_4$ and we obtain
$$A_5=R_5[x]/(x^2-hx-t),$$
where the coefficient ring $R_4=\mathbb{Z}[h,t]$ and the comultiplication and counit are
$$\Delta(1)=1\otimes x+x\otimes 1-h1\otimes 1,$$
$$\Delta(x)=x\otimes x+t1\otimes 1,$$
$$\epsilon(1)=0,\quad \epsilon(x)=1.$$

In conclusion, any Frobenius algebra whose corresponding TQFT can be used to defind a link homology is realized by the Frobenius algebra $A_5$. In other words, any Khovanov-type homology can be recovered by the link homology induced by $A_5$.

\subsection{A necessary condition for Reidemeister move invariance}\label{A necessary condition for Reidemeister-II invariance}
In this subsection, we give a necessary condition to make such homology invariant under the second Reidemeister move, which is depicted on left side of Figure~\ref{fig:Reidemeister-II move and examples}.

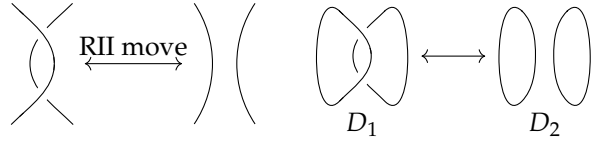
\begin{figure}[htbp]
  \centering 
  \begin{tikzpicture}[scale=0.8,use Hobby shortcut,baseline=0]
    \begin{knot}[clip width=3.7,]
    \strand (-0.5,0.5)..++(0.01,-0.01)..(0.2,-0.5)..++(-0.7,-1)..++(-0.01,-0.01);
    \strand (0.5,0.5)..++(-0.01,-0.01)..(-0.2,-0.5)..++(0.7,-1)..++(0.01,-0.01);
    \end{knot} 
    \draw[<->] (0.7,-0.5)--(2.3,-0.5);
    \node at (1.5,-0.2) {RII move};
    \draw (2.5,0.5)..(2.8,-0.5)..(2.5,-1.5) (3.5,0.5)..(3.2,-0.5)..(3.5,-1.5); 
  \end{tikzpicture}
  \quad \quad 
  \begin{tikzpicture}[scale=0.6,use Hobby shortcut, baseline=0]
    \begin{knot}[clip width=3.7,]
    \strand (-0.5,0.5)..++(0.01,-0.01)..(0.2,-0.5)..++(-0.7,-1)..++(-0.01,-0.01)..(-1,-0.5)..(-0.5,0.5)..++(0.01,-0.01);
    \strand (0.5,0.5)..++(-0.01,-0.01)..(-0.2,-0.5)..++(0.7,-1)..++(0.01,-0.01)..(1,-0.5)..(0.5,0.5)..++(-0.01,-0.01);
    \end{knot} 
    \draw[<->] (1.3,-0.5)--(2.7,-0.5);
    \draw (3.5,0.5)..++(0.01,-0.01)..(3.8,-0.5)..(3.5,-1.5)..++(-0.01,-0.01)..(3,-0.5)..(3.5,0.5)..++(0.01,-0.01) (4.5,0.5)..++(-0.01,-0.01)..(4.2,-0.5)..(4.5,-1.5)..++(0.01,-0.01)..(5,-0.5)..(4.5,0.5)..++(-0.01,-0.01); 
    \node at (0,-2) {$D_1$};
    \node at (4,-2) {$D_2$};
  \end{tikzpicture}
  \caption{Reidemeister-II move and examples}
  \label{fig:Reidemeister-II move and examples}
\end{figure}

Consider the two link diagrams $D_1, D_2$ on the right side of Figure~\ref{fig:Reidemeister-II move and examples}, both of which represent a trivial 2-component link. Fix an almost 2D-TQFT $\mathcal{F}^a$ with $\mathcal{F}^a(S^1)=A$. If the homology of this almost 2D-TQFT is invariant under the second Reidemeister move, then we should have $H(D_1;\mathcal{F}^a)\cong H(D_2;\mathcal{F}^a)$ up to a degree shift.
 
\begin{figure}[htbp]
  \centering
  \begin{tikzcd}
    &             &                                               &  & A\otimes A \arrow[rrd, "-m"] &  &             &   \\
C(D_1): & 0 \arrow[r] & A \arrow[rru, "\Delta"] \arrow[rrd, "\Delta"] &  & \oplus                       &  & A \arrow[r] & 0, \\
    &             &                                               &  & A\otimes A \arrow[rru, "m"]  &  &             &  
  \end{tikzcd}
  \caption{$C(D_1)$}
  \label{fig:$C(D_1)$}
\end{figure}

As $D_2$ is the disjoint union of two circles, its homology is $A\otimes A$, and the complex of $D_1$ is given in Figure \ref{fig:$C(D_1)$}. Denote $r=\mathrm{dim}A$, then the Euler characteristic of $C(D_1)$, $\chi(C(D_1))=r-2r^2+r=2r(1-r)$ should be equal or opposite to the dimension of $H(D_2)\cong A\otimes A$, which is $r^2$. Since $r\in\mathbb{N}$, we have $r=2$, and the homology group survive only in $H^1(D_1)\cong A\otimes A$. This tells us that the comultiplication $\Delta$ is injective and the multiplication $m$ is surjective. According to Theorem \ref{theorem1.1}, we obtain the following result.

\begin{lemma}\label{RII-lemma}
If the diagram homology $H(D; \mathcal{F}^a)$ induced by an almost 2D-TQFT $\mathcal{F}^a$ is invariant under the second Reidemeister move, then nearly Frobenius algebra corresponding to $\mathcal{F}^a$ is a Frobenius algebra.
\end{lemma}

\begin{remark}
The readers may wonder why we only consider the second Reidemeister move rather than the first and third ones. Actually, by considering the knot diagram which contains only one crossing point one can also induce that $\mathrm{dim}A=2$. By considering the knot diagram with only one positive crossing and the knot diagram with only one negative crossing, one can induce that the multiplication is surjective and the comultiplication is injective, respectively. However, if we wonder whether we can define a framed link homology invariant, i.e. the invariance under the first Reidemeister move is not necessary, we still need to consider the second Reidemeister move. So, even if we just want $H(D; \mathcal{F}^a)$ to be a framed link invariant (e.g. the Kauffman bracket), Lemma \ref{RII-lemma} still holds.
\end{remark}

\begin{theorem}
Given an almost 2D-TQFT $\mathcal{F}^a$ over a PID, then $\mathcal{F}^a$ induces a link homology if and only if the nearly Frobenius algebra corresponding to $\mathcal{F}^a$ is a Frobenius algebra realized by $A_5=\mathbb{Z}[h, t][x]/(x^2-hx-t)$.
\end{theorem}
\begin{proof}
According to the result in \cite{Kho2006}, when $A$ is a Frobenius algebra realized by $A_5$, $\mathcal{F}^a$ does induces a link homology.

Conversely, when $H(D; \mathcal{F}^a)$ defines a link homology theory, it is of course preserved by the second Reidemeister move. Now Theorem \ref{theorem1.1} and Lemma \ref{RII-lemma} tell us that such an almost 2D-TQFT is indeed a $2$D-TQFT. In this case, according to \cite{Kho2006} again, the corresponding nearly Frobenius algebra is a Frobenius algebra realized by $A_5$.
\end{proof}

This theorem tells us that in PID coefficient case, even if the unit and counit of the Frobenius algebra are not used in the construction of Khovanov-type homology, if it defines a link invariant, such maps always exist.

\end{document}